\documentclass[11pt]{article}

\usepackage[a4paper, top=3cm]{geometry} 

\usepackage{a4wide,epsfig}
\usepackage{amsfonts,amsmath}
\usepackage{graphicx}
\usepackage{amssymb}
\usepackage{hyperref}
\usepackage{url}
\usepackage{booktabs}
\usepackage{amsfonts}
\usepackage{nicefrac}
\usepackage{microtype}
\usepackage{graphicx}
\usepackage{doi}
\usepackage{tcolorbox}
\usepackage{amsthm}
\usepackage{dsfont}
\usepackage{mathtools}
\usepackage{tcolorbox}
\usepackage{amsmath}
\usepackage{physics}
\usepackage{bbm}
\usepackage{wasysym}
\usepackage{caption}
\usepackage{subcaption}
\usepackage{xcolor}
\definecolor{darkgreen}{rgb}{0,0.5,0}

\usepackage{pgfplots}
\usepgfplotslibrary{groupplots}
\usepackage{amsmath}
\usepackage{tikz}
\usepackage{amsfonts}
\usepackage{nicefrac}
\usepackage{soul}

\DeclareMathOperator{\Id}{Id}

\newcommand{\Ppt}{\Pi_{h_t}^{\partial_t^2}}
\newcommand{\Pg}{\Pi^{\nabla}_{h_{\bx}}}

\usepackage{oplotsymbl}

\newtheorem{theorem}{Theorem}[section]
\newtheorem{assumption}[theorem]{Assumption}
\newtheorem{lemma}[theorem]{Lemma}
\newtheorem{proposition}[theorem]{Proposition}
\newtheorem{corollary}[theorem]{Corollary}
\newtheorem{remark}[theorem]{Remark}

\DeclareFontFamily{U}{matha}{\hyphenchar\font45}
\DeclareFontShape{U}{matha}{m}{n}{
<-6> matha5 <6-7> matha6 <7-8> matha7
<8-9> matha8 <9-10> matha9
<10-12> matha10 <12-> matha12
}{}
\DeclareSymbolFont{matha}{U}{matha}{m}{n}
\DeclareFontFamily{U}{mathx}{\hyphenchar\font45}
\DeclareFontShape{U}{mathx}{m}{n}{
<-6> mathx5 <6-7> mathx6 <7-8> mathx7
<8-9> mathx8 <9-10> mathx9
<10-12> mathx10 <12-> mathx12
}{}
\DeclareSymbolFont{mathx}{U}{mathx}{m}{n}
\DeclareMathDelimiter{\vvvert} {0}{matha}{"7E}{mathx}{"17}%

\pgfplotsset{
  log x ticks with fixed point/.style={
      xticklabel={
        \pgfkeys{/pgf/fpu=true}
        \pgfmathparse{exp(\tick)}%
        \pgfmathprintnumber[fixed relative, precision=3]{\pgfmathresult}
        \pgfkeys{/pgf/fpu=false}
      }
  },
  log y ticks with fixed point/.style={
      yticklabel={
        \pgfkeys{/pgf/fpu=true}
        \pgfmathparse{exp(\tick)}%
        \pgfmathprintnumber[fixed relative, precision=3]{\pgfmathresult}
        \pgfkeys{/pgf/fpu=false}
    }
  }
}

\newcommand{\logLogSlopeTriangle}[5]
{
    \pgfplotsextra
    {
        \pgfkeysgetvalue{/pgfplots/xmin}{\xmin}
        \pgfkeysgetvalue{/pgfplots/xmax}{\xmax}
        \pgfkeysgetvalue{/pgfplots/ymin}{\ymin}
        \pgfkeysgetvalue{/pgfplots/ymax}{\ymax}

        \pgfmathsetmacro{\xArel}{#1}
        \pgfmathsetmacro{\yArel}{#3}
        \pgfmathsetmacro{\xBrel}{#1-#2}
        \pgfmathsetmacro{\yBrel}{\yArel}
        \pgfmathsetmacro{\xCrel}{\xArel}

        \pgfmathsetmacro{\lnxB}{\xmin*(1-(#1-#2))+\xmax*(#1-#2)} 
        \pgfmathsetmacro{\lnxA}{\xmin*(1-#1)+\xmax*#1} 
        \pgfmathsetmacro{\lnyA}{\ymin*(1-#3)+\ymax*#3} 
        \pgfmathsetmacro{\lnyC}{\lnyA+#4*(\lnxA-\lnxB)}
        \pgfmathsetmacro{\yCrel}{(\lnyC-\ymin)/(\ymax-\ymin)} 

        \coordinate (A) at (rel axis cs:\xArel,\yArel);
        \coordinate (B) at (rel axis cs:\xBrel,\yBrel);
        \coordinate (C) at (rel axis cs:\xCrel,\yCrel);

        \draw[#5]   (A)-- node[pos=0.5,anchor=north] {}
                    (B)-- 
                    (C)-- node[pos=0.5,anchor=west,] {\scriptsize$\boldsymbol{#4}$}
                    cycle;
    }
}

\newcommand{\logLogSlopeTriangleDual}[6]
{
    \pgfplotsextra
    {
        \pgfkeysgetvalue{/pgfplots/xmin}{\xmin}
        \pgfkeysgetvalue{/pgfplots/xmax}{\xmax}
        \pgfkeysgetvalue{/pgfplots/ymin}{\ymin}
        \pgfkeysgetvalue{/pgfplots/ymax}{\ymax}

        \pgfmathsetmacro{\xArel}{#1}
        \pgfmathsetmacro{\yArel}{#3}
        \pgfmathsetmacro{\xBrel}{#1 - #2}
        \pgfmathsetmacro{\yBrel}{\yArel}
        \pgfmathsetmacro{\xCrel}{\xArel}

        \pgfmathsetmacro{\lnxB}{\xmin*(1-(#1 - #2)) + \xmax*(#1 - #2)} 
        \pgfmathsetmacro{\lnxA}{\xmin*(1-#1) + \xmax*#1} 
        \pgfmathsetmacro{\lnyA}{\ymin*(1-#3) + \ymax*#3} 
        \pgfmathsetmacro{\lnyC}{\lnyA + #4*(\lnxA - \lnxB)}
        \pgfmathsetmacro{\yCrel}{(\lnyC - \ymin)/(\ymax - \ymin)} 

        \coordinate (A) at (rel axis cs:\xArel,\yArel);
        \coordinate (B) at (rel axis cs:\xBrel,\yBrel);
        \coordinate (C) at (rel axis cs:\xCrel,\yCrel);
        \coordinate (M) at ($(B)!0.5!(C)$);

        \draw[#5] (A) -- (B);
        \draw[#6] (B) -- (C);
        \draw[#5] (C) -- (A);

        \node[xshift=20pt] at (M) (N) {};

        \begin{scope}
            \clip ($(N)+(-2pt,-4pt)$) rectangle ($(N)+(0pt,4pt)$);
            \node at (N) {\scriptsize\textcolor{#5}{$\boldsymbol{#4}$}};
        \end{scope}
        \begin{scope}
            \clip ($(N)+(0pt,-4pt)$) rectangle ($(N)+(2pt,4pt)$);
            \node at (N) {\scriptsize\textcolor{#6}{$\boldsymbol{#4}$}};
        \end{scope}
    }
}

\DeclarePairedDelimiterX{\normiii}[1]
{\vvvert}
{\vvvert}
{\ifblank{#1}{\:\cdot\:}{#1}}

\renewcommand{\phi}{\varphi}

\newcommand{\bx}{\boldsymbol{x}}
\newcommand{\bp}{\boldsymbol{p}}

\renewcommand{\L}{\mathcal{L}}
\newcommand{\A}{\mathcal{A}}
\newcommand{\R}{\mathbb{R}}
\renewcommand{\P}{\mathbb{P}}
\newcommand{\N}{\mathbb{N}}

\newcommand{\bh}{\boldsymbol{h}}

\title{
Intrinsic unconditional stability in space--time isogeometric approximation of the acoustic wave equation\\ in second-order formulation}
\author{Matteo~Ferrari,~Ilaria~Perugia \vspace{0.5cm}}

\date{
Faculty of Mathematics, Universität Wien, Vienna, Austria \vspace{0.15cm}
}

\begin{document}
\maketitle

\begin{abstract}
\noindent
We present a novel space–time isogeometric discretization of the acoustic wave equation in second-order formulation that is intrinsically unconditionally stable. The method relies on a variational framework inspired by [Walkington 2014], with an exponential weight introduced in the time integrals. Conformity requires at least $C^1$ regularity in time and $C^0$ in space. The approximation in time is carried out using spline functions. The unconditional stability of the space--time method for conforming discrete spaces arises naturally from the variational structure itself, rather than from any artificial stabilization mechanisms. The analysis of an associated ordinary differential equation problem in time yields error estimates with respect to the mesh size that are suboptimal by one order in standard Sobolev norms. However, for certain choices of approximation spaces, it achieves quasi-optimal estimates. In particular, we prove this for $C^1$-regular splines of even polynomial degree, and provide numerical evidence suggesting that the same behavior holds for splines with maximal regularity, irrespective of the degree. The error analysis is extended to the full space--time problem with tensor-product approximation spaces. Numerical results are provided to support the theoretical findings and demonstrate the sharpness of the estimates. \vspace{0.2cm}
\end{abstract}



\section{Introduction}

\noindent
The wave equation is a fundamental mathematical model used to describe phenomena such as sound, light and vibration. In this paper, we propose and analyze a space--time Galerin method for its discretization, which makes use of spline functions in time. This method is based on a second-order formulation. It is inspired by the Continuous Galerkin-Discontinuous Galerkin (CG-DG) method analyzed in \cite{Walkington2014} (see also~\cite{DongMascottoWang2024} and \cite{GomezNikolic2024}) but it is conforming and requires discrete spaces with at least~$C^1$ regularity in time. 

\noindent
We consider a space--time cylinder~$Q_T := \Omega \times (0,T)$, where~$\Omega \subset \R^d$ ($d = 1, 2, 3$) is a bounded Lipschitz domain with boundary~$\Gamma := \partial \Omega$, and~$T > 0$ a final propagation time. Given a source term~$F \in L^2(Q_T)$ and a wave velocity~$c = c(\bx) \ge c_0 >0$ for some $c_0 > 0$ independent of the time variable~$t$, with~$c \in L^\infty(\Omega)$, as model problem we consider the following Dirichlet problem for the wave equation:
\begin{equation} \label{eq:1}
	\begin{cases}
		\partial_t^2 U(\bx,t) - \text{div}(c^2(\bx) \nabla_{\bx} U(\bx,t)) = F(\bx,t) & (\bx,t) \in Q_T, 
		\\ U(\bx,t) = 0 & (\bx,t) \in \Gamma \times (0, T),
		\\ U(\bx,0) = \partial_t U(\bx,0) = 0 & \bx \in \Omega.
	\end{cases}
\end{equation}
Expanding the solution to~\eqref{eq:1} as a series of eigenfunctions of the spatial diffusion operator naturally leads to considering
the following initial value problem for a second-order ordinary differential equation (ODE):
\begin{equation} \label{eq:2}
    \begin{cases}
    \partial_t^2 u(t) + \mu u(t) = f(t) & t \in (0,T),
    \\ u(0) = \partial_t u(0) = 0,
    \end{cases}
\end{equation}
with a real parameter~$\mu>0$, and~$f \in L^2(0,T)$. We first focus on studying a discretization of the ODE~\eqref{eq:2} that is uniformly stable with respect to the parameter~$\mu$, analyzing its stability and 
convergence properties. Based on that, we then introduce and analyze a space--time, tensor-product discretization of the PDE problem~\eqref{eq:1}.

\noindent
This work falls within the very active line of research on space--time methods for the wave equation that are provably unconditionally stable (see, e.g., \cite{Walkington2014, Zank2021, HenningPalittaSimonciniUrban2022, LoscherSteinbachZank2023, BignardiMoiola2023, FraschiniLoliMoiolaSangalli2023, DongMascottoWang2024, Gomez2025, FuhrerGonzalezKarkulik2025, FerrariPerugiaZampa2025}). These developments are of particular importance, as stability guarantees play a crucial role in the design of efficient solvers, \emph{a~posteriori} error estimators and adaptivity, and strategies for complexity reduction.

\noindent
Various stable finite element methods have been proposed for~\eqref{eq:1}. One approach is to multiply the corresponding ODE~\eqref{eq:2} by a test function with a zero final condition, and then, after \emph{performing integration by parts}, impose the zero initial condition for the derivative in a weak sense. Then, methods that are unconditionally stable typically rely on adding non-consistent stabilization terms, which depend on the discrete space \cite{SteinbachZank2019, FraschiniLoliMoiolaSangalli2023, FerrariFraschini2024}, inserting a projection in time within the mass (ODE) or stiffness (PDE) term \cite{Zank2021}, or modifying the test space with a suitable operator, e.g., the Modified Hilbert transform \cite{LoscherSteinbachZank2023}, see also \cite{SteinbachZank2020}. An alternative approach \textit{avoids integrating by parts}. Within this approach, methods are classified into first-order (with the auxiliary variable~$v := \partial_t u$) and second-order. Various first-order methods are known to be unconditionally stable~\cite{BalesLasiecka1994,FrenchPeterson1996,FerrariFraschiniLoliPerugia2024, Gomez2025, FerrariPerugiaZampa2025}. However, they result in a system where the number of variables is doubled compared to a second-order scheme. 

\noindent
The method we propose involves multiplying the ODE in~\eqref{eq:2} by a test function with a zero initial condition, modified by a suitable operator that makes the bilinear form coercive in~$H^1(0,T)$, with a coercivity constant independent on the parameter~$\mu$. This transformation operator does not alter the support of the test functions and yields a variational formulation with exponential-weighted inner products. In that, our approach is related to the CG-DG method of~\cite{French1993}. To enforce the zero initial condition for the first derivative, we add a term involving the pointwise evaluation of the trial and test functions at~$t = 0$. Notably, the method remains consistent, as the additional term vanishes for the exact solution. Both trial and test functions are discretized with splines of regularity at least~$C^1$. Based on that, we introduce a space--time variational formulation of the PDE problem~\eqref{eq:1} and its tensor-product discretization. Compared to the spline-based methods in~\cite{FraschiniLoliMoiolaSangalli2023,FerrariFraschini2024}, which are based on~\cite{SteinbachZank2019}, the unconditional stability here is inherent to the variational structure itself, with no need for additional stabilization terms. The focus of this contribution is on the theoretical properties of the method. Computational aspects, such as its efficient implementation and its
reformulation as a time-marching scheme \cite{Tani2017, LangerZank2021, LoliSangalli2025} are beyond the scope of this work.

\noindent
Our approach shares conceptual similarities with the conforming Petrov–Galerkin space--time ultra-weak formulation presented in \cite{HenningPalittaSimonciniUrban2022}. In that work, the construction does not involve separating the variables: the test functions are~$C^1$-splines in space and time, while the trial functions are chosen to be discontinuous piecewise polynomials, to ensure discrete inf-sup stability. Another recent conforming inf-sup stable method has  been proposed in \cite{FuhrerGonzalezKarkulik2025} based on a least-squares approach, and a conforming coercive one (in space and time) in \cite{BignardiMoiola2023} based on Morawetz multipliers.
\medskip

\noindent
\textbf{Outline.}
The paper is structured as follows. 
In Section~\ref{sec:2}, we introduce and analyze a variational formulation for the ODE problem~\eqref{eq:2} and prove its coercivity. Then, in Section~\ref{sec:3}, we introduce a conforming discretization and prove coercivity and $h$-version error estimates, which are suboptimal by one order in standard Sobolev norms, when employing discrete spaces that satisfy standard approximation properties and inverse inequalities. Furthermore, we improve this result for discrete spaces that also admit a specific projection operator with quasi-optimal approximation properties, which is the case for $C^1$ splines with even polynomial degree. The related proofs are postponed to the two appendices. We demonstrate numerically the sharpness of the coercivity estimate and convergence rates. 
In Section~\ref{sec:4}, we extend the proposed method for the ODE to the full wave problem~\eqref{eq:1}, and prove its intrinsic stability. Finally, in Section \ref{sec:5}, building on the analysis of the ODE in Section \ref{sec:3}, we derive error estimate for the full tensor product space--time method, that are always quasi-optimal in space, and quasi-optimality in time is achieved under the same conditions as in the ODE case. The unconditional stability and error estimates are also validated numerically.

\section{Variational formulation of the ODE problem} \label{sec:2}

In this section, we present and analyze a new variational formulation of \eqref{eq:2}. Motivated by~\cite{Walkington2014}, we consider the continuous variational formulation: find~$u \in H^2_{0,\bullet}(0,T)$\footnote{Throughout this paper, we use standard notation for differential operators, function spaces, and norms; see, e.g.,~\cite{Brezis2010}.}
such that
\begin{equation} \label{eq:3}
    (\partial_t^2 u, \partial_t w)_{L^2(0,T)} + \partial_t u(0) \partial_t w(0) + \mu (u, \partial_tw)_{L^2(0,T)} = (f,\partial_t w)_{L^2(0,T)}
\end{equation}
for all~$w \in H_{0,\bullet}^2(0,T)$, where the Sobolev space~$H_{0,\bullet}^2(0,T)$ is defined as
\begin{equation*}
    H_{0,\bullet}^2(0,T) := \{u \in H^2(0,T) : u(0)=0\}.
\end{equation*}
Formulation~\eqref{eq:3} represents the global scheme corresponding to \cite[Equation (3.1)]{Walkington2014} for functions that are at least $C^1$-smooth.
\begin{remark}
Since~$\partial_t : H_{0,\bullet}^2(0,T) \to H^1(0,T)$ is a bijection,~\eqref{eq:3} is equivalent to: find $u \in H^2_{0,\bullet}(0,T)$ such that
\begin{equation} \label{eq:4}
    (\partial_t^2 u, v)_{L^2(0,T)} + \partial_t u(0) v(0) + \mu (u, v)_{L^2(0,T)} = ( f,v)_{L^2(0,T)}
\end{equation}
for all~$v \in H^1(0,T)$. In both~\eqref{eq:3} and~\eqref{eq:4}, the test functions are required to have regularity~$H^2$ and~$H^1$, respectively, to ensure that the traces at~$t=0$ are well-defined.
\end{remark}
\noindent
At least one solution of~\eqref{eq:3} exists, as the unique strong solution of~\eqref{eq:2} also satisfies~\eqref{eq:3}. It remains to establish the uniqueness.

\noindent
Define the bilinear form~$a_\mu : H_{0,\bullet}^2(0,T) \times H_{0,\bullet}^2(0,T) \to \R$ as
\begin{equation} \label{eq:5}
    a_\mu(u,w) := (\partial_t^2 u, \partial_t w)_{L^2(0,T)} + \partial_t u(0) \partial_t w(0) + \mu (u, \partial_tw)_{L^2(0,T)}.
\end{equation}
\begin{lemma} \label{lem:22}
For all~$u \in H_{0,\bullet}^2(0,T)$, the bilinear form defined in~\eqref{eq:5} satisfies
\begin{equation} \label{eq:6}
    a_\mu\big(u,\int_0^\bullet e^{-s/T} \partial_s u(s) \, \dd s\big) \ge \frac{1}{2eT} \left(\| \partial_t u\|^2_{L^2(0,T)} + \mu \| u \|^2_{L^2(0,T)}\right), 
\end{equation}
from which we also deduce the stability estimates
\begin{equation} \label{eq:7}
     \sup_{0 \neq  w \in H_{0,\bullet}^2(0,T)} \frac{a_\mu(u,w)}{\| \partial_t w\|_{L^2(0,T)}} \ge \frac{1}{2eT} \| \partial_t u\|_{L^2(0,T)}.
\end{equation}
\end{lemma}
\begin{proof}
For a given~$u \in H_{0,\bullet}^2(0,T)$ and a constant~$\kappa>0$, let consider~$w \in H_{0,\bullet}^2(0,T)$ such that~$\partial_t w(t) = e^{-\kappa t} \partial_t u(t)$. First, note that, since~$e^{-\kappa T} \le e^{-\kappa t} \le 1$ for~$t \in [0,T]$, we have
\begin{equation}\label{eq:8}
    \frac{1}{e^{\kappa T}} \| \partial_t u\|_{L^2(0,T)} \le \| \partial_t w\|_{L^2(0,T)} \le \| \partial_t u\|_{L^2(0,T)}.
\end{equation}
Then, using~$\partial_t u \partial_t^2 u = \frac{1}{2} \partial_t \left(|\partial_t u|^2\right)$, we calculate
\begin{equation} \label{eq:9}
\begin{aligned}
    (\partial_t^2 u, \partial_t w)_{L^2(0,T)} & = \int_0^T \partial_t^2 u(t) e^{-\kappa t} \partial_t u(t) \, \dd t 
    \\ & = \frac{\kappa}{2} \int_0^T e^{-\kappa t} |\partial_t u(t)|^2 \, \dd t + \frac{1}{2e^{\kappa T}} | \partial_t 
    u(T)|^2 - \frac{1}{2}| \partial_t u(0)|^2. 
\end{aligned}
\end{equation}
Similarly, using~$u \partial_t u = \frac{1}{2} \partial_t \left(|u|^2\right)$ and~$u(0)=0$, we calculate
\begin{align} \label{eq:10}
    (u, \partial_tw)_{L^2(0,T)} & = \int_0^T u(t) e^{-\kappa t} \partial_t u(t) \, \dd t = \frac{\kappa}{2}
    \int_0^T e^{-\kappa t} |u(t)|^2 \, \dd t + \frac{1}{2e^{\kappa T}} |u(T)|^2.
\end{align}
Then, combining~\eqref{eq:9} and~\eqref{eq:10}, we deduce
\begin{equation} \label{eq:11}
\begin{aligned}
    a_\mu(u,w) & \ge \frac{\kappa}{2} \int_0^T e^{-\kappa t} |\partial_t u(t)|^2 \, \dd t  + \mu \frac{\kappa}{2} \int_0^T e^{-\kappa t} |u(t)|^2 \, \dd t
     \\ & \ge \frac{\kappa}{2e^{\kappa T}} \left( \| \partial_t u\|^2_{L^2(0,T)} + \mu \| u\|^2_{L^2(0,T)} \right).
\end{aligned}
\end{equation}
The optimal choice for~$\kappa$ is then~$\kappa=1/T$. With this choice, we obtain~\eqref{eq:6} from which, using~\eqref{eq:8}, we also deduce~\eqref{eq:7}.
\end{proof}
\noindent
From~\eqref{eq:7}, we conclude uniqueness of the solution of~\eqref{eq:3} and, for this unique solution~$u \in H_{0,\bullet}^2(0,T)$, the stability estimate
\begin{equation} \label{eq:12}
    \| \partial_t u \|_{L^2(0,T)} \le 2eT \| f \|_{L^2(0,T)}.
\end{equation}
\begin{remark}
Estimate~\eqref{eq:11} deteriorates as $\kappa \to 0^+$. Therefore, the proof of Lemma~\ref{lem:22} does not guarantee that a direct discretization of problem \eqref{eq:3} is stable. For this reason, we will base our numerical method on a different writing of formulation~\eqref{eq:3}. Nevertheless, we have tested numerically a direct discretization of~\eqref{eq:3} with splines and it seems to be unconditionally stable. This will be investigated in future work. A similar situation arises in space--time boundary element methods for the wave equation, where the single-layer operator is proven to be coercive when composed with a Laplace transform with strictly positive argument \cite{BambergerHaDoung1986}. However, in practical computations, the parameter is typically set to zero (see~\cite[Theorem 2.4]{CostabelSayas2004}).
\end{remark}
\noindent
The proof of Lemma~\ref{lem:22} suggests to introduce the operator~$\L_T : H^2_{0,\bullet}(0,T) \to H^2_{0,\bullet}(0,T)$ defined as
\begin{equation} \label{eq:13}
    \L_T w(t) : = \int_0^t e^{-s/T} \partial_s w(s) \, \dd s.
\end{equation}
In the following proposition, we collect some properties of this operator, whose proofs are straightforward.
\begin{proposition} \label{prop:24}
The operator~$\L_T$ defined in~\eqref{eq:13} satisfies the following properties:
\begin{enumerate}
\item~$\L_T$ is actually well defined as an operator from~$H^1(0,T)$ to~$H^1(0,T)$,
\item for all~$v \in H^1(0,T)$,
\begin{equation*}
    \partial_t\L_T v(t)=e^{-t/T}\partial_t v(t),
\end{equation*}
\item~$\L_T$ is invertible in~$H^2_{0,\bullet}(0,T)$ with inverse
\begin{equation*}
    \L_T^{-1} u(t) = \int_0^t e^{s/T} \partial_s u(s) \, \dd s,
\end{equation*}
\item for all~$u \in H^2(0,T)$,
\begin{equation*}
    (\partial_t^2 u, \partial_t \L_T u)_{L^2(0,T)} + |\partial_t u(0)|^2 = \frac{1}{2T} (\partial_t u, \partial_t \L_T u)_{L^2(0,T)} + \frac{1}{2e} | \partial_t u(T)|^2 + \frac{1}{2} |\partial_t u(0)|^2,
\end{equation*}
\item for all~$v \in H^1(0,T)$ with $v(0)=0$,
\begin{equation*}
    (v, \partial_t \L_T v)_{L^2(0,T)} = \frac{1}{2T} \underbrace{(v, \L_T \int_0^\bullet v(\sigma) \, \dd \sigma)_{L^2(0,T)}}_{=\int_0^T e^{-t/T}|v(t)|^2 \, \dd t} + \frac{1}{2e} |v(T)|^2,
\end{equation*}
\item for all~$v \in H^1(0,T)$,
\begin{equation*}
    \frac{1}{e} \| \partial_t v \|^2_{L^2(0,T)} \le (\partial_t v, \partial_t \L_T v)_{L^2(0,T)} \le \| \partial_t v \|^2_{L^2(0,T)},
\end{equation*}
\item for all~$\lambda \in L^2(0,T)$,
\begin{equation*}
    \frac{1}{e} \| \lambda \|^2_{L^2(0,T)} \le \underbrace{(\lambda, \L_T \int_0^\bullet \lambda(\sigma) \, \dd \sigma)_{L^2(0,T)}}_{=\int_0^T e^{-t/T}|\lambda(t)|^2 \, \dd t} \le \|  \lambda \|^2_{L^2(0,T)},
\end{equation*}
\item for all~$v \in H^1(0,T)$,
\begin{equation*}
    \frac{1}{e} \| \partial_t v \|_{L^2(0,T)} \le \| \partial_t \L_T v\|_{L^2(0,T)} \le \| \partial_t  v \|_{L^2(0,T)}.
\end{equation*}
\end{enumerate}
\end{proposition}
\medskip

\noindent
We rewrite the variational problem~\eqref{eq:3} as follows:
\begin{tcolorbox}[
    colframe=black!50!white,
    colback=blue!5!white,
    boxrule=0.5mm,
    sharp corners,
    boxsep=0.5mm,
    top=0.5mm,
    bottom=0.5mm,
    right=0.25mm,
    left=0.1mm
]
    \begingroup
    \setlength{\abovedisplayskip}{0pt}
    \setlength{\belowdisplayskip}{0pt}
    \begin{equation} \label{eq:14}
    \text{find } u \in H^2_{0,\bullet}(0,T)     \text{ such that }
        a_\mu(u, \L_T w) = (f, \partial_t \L_T w)_{L^2(0,T)} \, \, \text{for all~} \, w \in H_{0,\bullet}^2(0,T).
    \end{equation}
    \endgroup
\end{tcolorbox}
\noindent
From a continuous point of view,~\eqref{eq:3} and~\eqref{eq:14} are the same problem. However, by virtue of Lemma~\ref{lem:22}, the bilinear form~$a_\mu(\cdot,\L_T \,\cdot)$ associated with~\eqref{eq:14} is coercive in the~$H^1$ norm, namely
\begin{equation} \label{eq:15}
    a_\mu(u,\L_T u) \ge \frac{1}{2eT}  \| \partial_t u\|^2_{L^2(0,T)} \quad \text{for all~} u \in H^2_{0,\bullet}(0,T).
\end{equation}
We highlight that the coercivity constant in~\eqref{eq:15} is independent of~$\mu$.

\noindent
In the next result, we show that continuity is only guaranteed in the~$H_{0,\bullet}^2(0,T)$ norm defined as
\begin{equation} \label{eq:16}
    \| u \|_{H^2_{0,\bullet}(0,T)} := \| \partial_t^2 u \|_{L^2(0,T)} + \frac{1}{T} \,
    \| \partial_t u \|_{L^2(0,T)},
\end{equation}
where the weight~$T^{-1}$ in front of the second term is introduced to ensure dimensional consistency. 
\begin{proposition} \label{prop:25}
For all~$u, w \in H_{0,\bullet}^2(0,T)$, we have 
\begin{equation} \label{eq:17}
    a_\mu(u,\L_T w) \le \left(1+\mu \frac{2T^{2}}{\pi}\right) \| u\|_{H_{0,\bullet}^2(0,T)} \|\partial_t w\|_{L^2(0,T)} + | \partial_t u(0) | | \partial_t w(0) |,
\end{equation}
or also
\begin{equation} \label{eq:18}
    a_\mu(u,\L_T w) \le \left(1+\mu \frac{2T^{2}}{\pi}\right) \| \partial_t u\|_{L^2(0,T)} \| w\|_{H_{0,\bullet}^2(0,T)} + \frac{1}{e} \, | \partial_t u(T) | | \partial_t w(T) |.
\end{equation}
\end{proposition}
\begin{proof}
With the Cauchy-Schwarz inequality and the properties in Proposition~\ref{prop:24},
we readily obtain 
\begin{align*}
    a_\mu(u,\L_T w) & = (\partial_t^2 u, \partial_t \L_T w)_{L^2(0,T)} + \partial_t u(0) \partial_t \L_T w(0) + \mu ( u, \partial_t \L_T w)_{L^2(0,T)} 
    \\ & \le \|\partial_t^2 u\|_{L^2(0,T)} \|\partial_t 
    \L_T w\|_{L^2(0,T)} + |\partial_t u(0)| |\partial_t w(0)| + \mu \|u \|_{L^2(0,T)} \|\partial_t \L_T w\|_{L^2(0,T)}
    \\ & \le \|\partial_t^2 u\|_{L^2(0,T)} \|\partial_t w\|_{L^2(0,T)} + |\partial_t u(0)| |\partial_t w(0)| + \mu \frac{2T}{\pi} \| \partial_t u \|_{L^2(0,T)} \|\partial_t w\|_{L^2(0,T)}, 
\end{align*}
where, in the last term, we used the Poincar\'e inequality (see \cite{PayneWeinberger1960}) with sharp constant $\|u \|_{L^2(0,T)}\le \frac{2T}{\pi}\| \partial_t u \|_{L^2(0,T)}$. This proves~\eqref{eq:17}. The continuity estimate~\eqref{eq:18} is obtained integrating by parts the first term of~$a_\mu$, and proceeding as in the estimate of~\eqref{eq:17}. 
\end{proof}

\begin{remark}
From the continuity property~\eqref{eq:17}, we also deduce the following continuity estimate in the~$H_{0,\bullet}^2(0,T)$ norm:
\begin{equation} \label{eq:19}
    a_\mu(u,\L_T w) \le 2T \left(1+\mu \frac{T^2}{\pi}\right)
    \| u\|_{H_{0,\bullet}^2(0,T)} \| w\|_{H_{0,\bullet}^2(0,T)}\qquad \text{for all~} u, w \in H_{0,\bullet}^2(0,T).
\end{equation}
In order to do so, we only need to estimate~$| \partial_t u(0)|$. By the fundamental theorem of calculus and the Cauchy-Schwarz inequality, we obtain
\begin{equation*}
\begin{split}
    |\partial_t u (0)| = \left\lvert-\int_0^t \partial_s^2 u (s) \, \dd s + \partial_t u (t)\right\rvert 
    \le t^{\nicefrac{1}{2}} \| \partial_t^2 u\|_{L^2(0,T)} + |\partial_t u (t)| \quad \text{for all~} t \in [0,T],
\end{split}    
\end{equation*}
so that
\begin{equation*}
    \int_0^T |\partial_t u (0)| \, \dd t \le \int_0^T  t^{\nicefrac{1}{2}} \, \dd t \, \| \partial_t^2 u\|_{L^2(0,T)} + \int_0^T |\partial_t u (t)| \, \dd t. 
\end{equation*}
Then, using again the Cauchy-Schwarz inequality, we get
\begin{equation} \label{eq:20}
     |\partial_t u (0)|  \le \frac{2}{3}T^{\frac{1}{2}}  \| \partial_t^2 u\|_{L^2(0,T)} + T^{-\frac{1}{2}} \|\partial_t u \|_{L^2(0,T)}\le T^{\frac{1}{2}}\|u\|_{H_{0,\bullet}^2(0,T)},
\end{equation}
from which we deduce~\eqref{eq:19}.
\end{remark}
\section{Discretization of the ODE problem} \label{sec:3}
In this section, we present a discretization of formulation~\eqref{eq:14} that is stable and well-suited for practical implementation.

\noindent
Let~$S_h^p(0,T) \subset H_{0,\bullet}^2(0,T)$ be a discrete space, depending on a mesh parameter~$h>0$ and a polynomial degree~$p \in \N_0$. We consider 
the conforming discretization of the variational formulation~\eqref{eq:14}:
\begin{tcolorbox}[
    colframe=black!50!white,
    colback=blue!5!white,
    boxrule=0.5mm,
    sharp corners,
    boxsep=0.5mm,
    top=0.5mm,
    bottom=0.5mm,
    right=0.25mm,
    left=0.1mm
]
    \begingroup
    \setlength{\abovedisplayskip}{0pt}
    \setlength{\belowdisplayskip}{0pt}
    \begin{equation} \label{eq:21}
    \text{find~} u_h \in S_h^p(0,T)~\text{such that~}a_\mu(u_h, \L_T w_h) = (f, \partial_t \L_T w_h )_{L^2(0,T)} \,\, \text{for all} \,\, w_h \in S_h^p(0,T).
\end{equation}
    \endgroup
\end{tcolorbox}
\noindent
Recalling the definitions of~$a_\mu$ in~\eqref{eq:5} and of~$\L_T$ in~\eqref{eq:13}, the explicit expression of the discrete formulation~\eqref{eq:21} is
\begin{equation} \label{eq:22}
\begin{aligned}
    \int_0^T\partial_t^2 u_h(t) \partial_t w_h(t) e^{-t/T} \, \dd t & +\partial_t u_h(0)\partial_t w_h(0) 
    \\ & +\mu\int_0^T u_h(t) \partial_t w_h(t) e^{-t/T} \, \dd t = \int_0^T f(t) \partial_t w_h(t) e^{-t/T} \, \dd t.
\end{aligned}
\end{equation}
\noindent
With the coercivity property~\eqref{eq:15} in the~$H^1$ norm transferring directly to the discrete level, we immediately have the well-posedness of the discrete problem~\eqref{eq:21}.
\begin{remark}
From an implementation perspective, the presence of the operator~$\L_T$ is not problematic. First of all, $\L_T$ preserves the structure of the involved matrices. This is because, if a function~$w_h \in S_h^p(0,T)$ has compact support, then~$\partial_t \L_T w_h$ maintains the same compact support. Additionally, standard techniques allow for the design of efficient Gaussian quadrature formulas for integrals of the form~$\int_{t_j}^{t_{j+1}} q(t) e^{-t/T} \dd t$, where~$q(t)$ is a polynomial. Consider, for example, a uniform mesh with nodes~$t_j = jh$ for~$j = 0, \ldots, N$, and~$h = T/N$. Suppose we have a quadrature formula that is exact up to a certain order of accuracy for computing the integral~$\int_0^h q(t) e^{-t/T} \dd t$. This formula can be modified to compute the integral on mesh intervals as follows:
\begin{equation*}
    \int_{t_j}^{t_{j+1}} q(t) e^{-t/T} \, \dd t = e^{-t_j/T} \int_0^h q(t + t_j) e^{-t/T} \, \dd t.
\end{equation*}
These integrals are computed exactly, without any quadrature error, by using a sufficient number of quadrature points in each subinterval. This approach justifies omitting the quadrature error in the subsequent stability and error estimates. Based on our numerical experience, we also observe that standard Gauss-Legendre quadrature does not change the numerical results and performs well in practice.
\end{remark}
\noindent
We make the following assumption.
\begin{assumption} \label{ass:32}
The finite dimensional spaces~$S_h^p(0,T) \subset H_{0,\bullet}^2(0,T)$ satisfy the following inverse inequalities and approximation properties:
\begin{itemize}
\item[i)] there exists a constant~$C_I>0$ depending on~$p$ but not on~$h$ such that
\begin{align}
     \label{eq:23} \| w_h \|_{H^2_{0,\bullet}(0,T)} & \le C_I h^{-1} \| \partial_t w_h \|_{L^2(0,T)} \quad \text{for all~} w_h \in S_h^p(0,T),
    \\ 
    \label{eq:24} \| \partial_t w_h \|_{L^2(0,T)} & \le C_I h^{-1} \| w_h \|_{L^2(0,T)} \quad \text{for all~} w_h \in S_h^p(0,T),
\end{align}
\item[ii)] there exists a constant~$C_a>0$ depending on~$p$ but not on~$h$ such that, for each $w \in H^2_{0,\bullet}(0,T) \cap H^{s+1}(0,T)$ and 
$\ell \le s \le p$, there is a function~$\widetilde{w}_h \in S_h^p(0,T)$ satisfying
\begin{align}
    \label{eq:25} \| \partial_t^\ell w - \partial_t^\ell \widetilde{w}_h \|_{L^2(0,T)} & \le C_a h^{s+1-\ell} \| \partial_t^{s+1} w \|_{L^2(0,T)}, \quad \ell = 0,1,2,
\end{align}
and, if $w \in W_\infty^{s+1}(0,T)$, also
\begin{align}
    \label{eq:26} \| \partial_t w - \partial_t \widetilde{w}_h \|_{L^\infty(0,T)} & \le C_a h^s \| \partial_t^{s+1} w \|_{L^\infty(0,T)}.
\end{align}
\end{itemize}
\end{assumption}
\noindent
Suitable choices for the spaces~$S_h^p(0,T)$ 
satisfying Assumption~\ref{ass:32} are the spaces generated by B-splines (see, e.g., \cite{PieglTiller2012, DeBoor2001}) with at least~$C^1$ regularity and vanishing at~$t=0$. Since B-splines are locally polynomials, for quasi-uniform meshes, the inverse inequalities~\eqref{eq:23} and~\eqref{eq:24} hold true. Furthermore, the approximation properties~\eqref{eq:25} and \eqref{eq:26} follow, e.g., from \cite[Theorem 49]{LycheManniSpeleers2018} (see also references therein), with~$\widetilde{w}_h\in S_h^p(0,T)$ a suitable quasi-interpolant of~$w\in C^0([0,T])$.
\medskip

\noindent
In the discrete case, the coercivity estimate~\eqref{eq:15} is shown to be sharp. Indeed, in Table~\ref{tab:1}, we report the discrete coercivity constants of the bilinear form~$a_\mu(\cdot, \L_T \cdot)$ in the~$H^1$ norm for spaces generated by B-splines of degree~$p = 2$ and~$p = 3$, with maximal regularity and vanishing at~$t = 0$.\footnote{All numerical test are performed with Matlab R2024a. The codes used for the numerical tests are available in the GitHub repository \cite{XTWavesExp}.} These constants are computed by solving a generalized eigenvalue problem for various values of~$T$,~$\mu$, and~$h$. Note that, when~$T = 1$, we have~$(2eT)^{-1} \approx 0.184$, and when~$T = 3$,~$(2eT)^{-1} \approx 0.061$. For small values of~$\mu$, the estimate in~\eqref{eq:15} is then sharp, as shown in Table~\ref{tab:1}. However, for large values of~$\mu$, the contribution of the second term in~\eqref{eq:6} is dominant, and by applying the inverse inequality~\eqref{eq:24}, we justify the quadratic decrease observed in Table~\ref{tab:1} when~$h$ is not yet sufficiently small.
\begin{table}[ht]
\centering
\begin{tabular}{lcccc|ccccc}
\hline 
$T=1$&  \multicolumn{2}{c}{$p=2$} & \multicolumn{2}{c|}{$p=3$} &~$T=3$ 
& \multicolumn{2}{c}{$p=2$} & \multicolumn{2}{c}{$p=3$} \\
\hline
~$h$ & \hspace{-0.1cm}$\mu=10$ & \hspace{-0.1cm}$\mu=10^5$ & \hspace{-0.1cm}$\mu=10$ & \hspace{-0.1cm}$\mu=10^5$ &~$h$ & \hspace{-0.1cm}$\mu=10$ & \hspace{-0.1cm}$\mu=10^5$ & \hspace{-0.1cm}$\mu=10$ & \hspace{-0.1cm}$\mu=10^5$ \\
\hline
    0.125 & 0.237 & 39.64 & 0.222 & 29.05 & 0.375 & 0.101 & 112.7 & 0.093 & 68.06
    \\ 0.063 & 0.206 & 10.39 & 0.200 & 10.33 & 0.188 & 0.075 & 26.57 & 0.072 & 20.10
    \\ 0.031 & 0.194 & 2.526 & 0.191 & 2.500 & 0.094 & 0.066 & 6.796 & 0.065 & 6.818
    \\ 0.016  & 0.188 & 0.730 & 0.187 & 0.729 & 0.047 & 0.063 & 1.645 & 0.063 & 1.650
    \\ 0.008 & 0.186 & 0.319 & 0.185 & 0.319 & 0.023 & 0.062 & 0.438 & 0.062 & 0.441
    \\ 0.004 & 0.185 & 0.219 & 0.185 & 0.219 & 0.012 & 0.062 & 0.153 & 0.062 & 0.154
    \\ 0.002 & 0.184 & 0.194 & 0.184 & 0.193 & 0.006 & 0.061 & 0.084 & 0.061 & 0.085
    \\ 0.001 & 0.184 & 0.187 & 0.184 & 0.187 & 0.003 & 0.061 & 0.067 & 0.061 & 0.067 \\
\hline
\end{tabular}
\caption{Discrete coercivity constant of~$a_\mu(\cdot,\L_T \cdot)$ in the~$H^1$ norm (see~\eqref{eq:15}) with maximal regularity splines of degree~$p=2,3$, by varying~$\mu=10,10^5$,~$T=1,3$, and the mesh size~$h$.}\label{tab:1}
\end{table}

\begin{remark}\label{rem:33}
Using the inverse inequality~\eqref{eq:23}, coercivity is established at the discrete level also in the~$H^2_{0,\bullet}(0,T)$ norm, although with a constant that depends on~$h$. More, precisely, we have
\begin{equation} \label{eq:27}
    a_\mu(u_h,\L_T u_h) \ge \frac{h^2}{2eT C_I^2}  \|  u_h\|^2_{H^2_{0,\bullet}(0,T)} \quad \text{for all~} u_h \in S_h^p(0,T).
\end{equation}
In experiments not reported here but available in the repository~\cite{XTWavesExp}, we verified that also estimate~\eqref{eq:27} is sharp.
\end{remark}

\subsection{Error analysis} \label{sec:31}
Recalling the continuity of the bilinear form~$a_\mu(\cdot,\L_T \cdot)$ established in Proposition~\ref{prop:25}, we derive error estimates with respect to the mesh size in the~$H^1$ norm.
\begin{theorem} \label{theo:34}
Let the discrete spaces~$S_h^p(0,T) \subset H^2_{0,\bullet}(0,T)$ satisfy Assumption~\ref{ass:32}. Let $u \in H^2_{0,\bullet}(0,T)$ be the unique solution of problem~\eqref{eq:3}, and let~$u_h \in S_h^p(0,T)$ be the unique solution of problem~\eqref{eq:21}. Then, if~$u \in H^2_{0,\bullet}(0,T) \cap H^{s+1}(0,T)$ for~$1 \le s \le p$, we have
\begin{equation} \label{eq:28}
    \| \partial_t u - \partial_t u_h \|_{L^2(0,T)} \le C_\mu h^{s-1} \| \partial_t^{s+1} u \|_{L^2(0,T)},
\end{equation}
for a positive constant~$C_\mu$ depending on~$p, T, \mu$ but not on~$h$.
\end{theorem}
\begin{proof}
For any~$w_h \in S_h^p(0,T)$, we use the triangle inequality to deduce
\begin{equation} \label{eq:29}
    \| \partial_t u - \partial_t u_h \|_{L^2(0,T)} \le \| \partial_t u - \partial_t w_h \|_{L^2(0,T)} + \|  \partial_t u_h - \partial_t w_h \|_{L^2(0,T)}.
\end{equation}
Then, from the coercivity property~\eqref{eq:15}, the linearity of~$\L_T$, and consistency, we get
\begin{align*}
    \frac{1}{2 e T} \| \partial_t u_h - \partial_t w_h \|^2_{L^2(0,T)} \le a_\mu(u_h - w_h, \L_T (u_h - w_h))
     = a_\mu(u - w_h, \L_T (u_h - w_h)),
\end{align*}
which, together with the continuity in~\eqref{eq:18}, the analogue of~\eqref{eq:20} for the final time, and the inverse inequality \eqref{eq:23}, gives
\begin{equation} \label{eq:30}
\begin{aligned}
    \frac{1}{2eT} \| \partial_t u_h - \partial_t w_h \|_{L^2(0,T)}^2 & \le \left( 1 + \mu \frac{2T^{2}}{\pi} \right) \| \partial_t u - \partial_t w_h \|_{L^2(0,T)} \| u_h - w_h \|_{H_{0,\bullet}^2(0,T)}
    \\  & \qquad\qquad\qquad + \,| (\partial_t u - \partial_t w_h )(T) |\, | (\partial_t u_h - \partial_t w_h )(T) |
    \\ & \le C_\mu h^{-1} \left(\| \partial_t u - \partial_t w_h \|_{L^2(0,T)} \right.
    \\ & \qquad\qquad\qquad \left. +\| \partial_t u - \partial_t w_h \|_{L^\infty(0,T)} \right) \| \partial_t u_h - \partial_t w_h \|_{L^2(0,T)}.
\end{aligned}
\end{equation}
Finally, combining~\eqref{eq:29} and~\eqref{eq:30} with the choice~$w_h = \widetilde{u}_h$ such that~\eqref{eq:25} is satisfied gives~\eqref{eq:28}.
\end{proof}
\noindent
From~\eqref{eq:28}, we immediately derive error estimates in the~$H^2$ norm.
\begin{corollary} \label{cor:35}Under the same assumptions as in Theorem~\ref{theo:34},
we have 
\begin{equation} \label{eq:31}
    \| u - u_h \|_{H^2_{0,\bullet}(0,T)} \le C_\mu h^{s-2} \| \partial_t^{s+1} u \|_{L^2(0,T)}.
\end{equation}
\end{corollary}
\begin{proof}
For any~$w_h \in S_h^p(0,T)$, we use the triangle inequality and the inverse inequality~\eqref{eq:23} to deduce
\begin{equation*}
\begin{aligned}
    \| u - u_h \|_{H_{0,\bullet}^2(0,T)} & \le \|u -  w_h \|_{H_{0,\bullet}^2(0,T)} + \|  u_h -  w_h \|_{H_{0,\bullet}^2(0,T)}
    \\ & \le \| u - w_h \|_{H_{0,\bullet}^2(0,T)} + C_I h^{-1} \|  \partial_t u_h - \partial_t w_h \|_{L^2(0,T)} 
    \\ & \le \| u - w_h \|_{H_{0,\bullet}^2(0,T)} + C_I h^{-1} \left( \| \partial_t u - \partial_t u_h \|_{L^2(0,T)} +  \|\partial_t  u - \partial_t w_h \|_{L^2(0,T)} \right).
\end{aligned}
\end{equation*}
Then, estimate~\eqref{eq:31} is obtained by selecting~$w_h = \widetilde{u}_h$ and using~\eqref{eq:28}, together with the approximation property~\eqref{eq:25}.
\end{proof}
\noindent
To establish error estimates in the~$L^2$ norm, we analyze the adjoint problem associated with~\eqref{eq:14}.
\begin{lemma} \label{lem:36}
Given~$f \in H^1(0,T)$ with~$f(T)=0$, consider the problem: find~$\phi \in H^2_{0,\bullet}(0,T)$ such that
\begin{equation} \label{eq:32}
    a_\mu(w,\L_T \phi) = (f, \partial_t \L_T w)_{L^2(0,T)} \quad \text{for all~} w \in H_{0,\bullet}^2(0,T).
\end{equation}
Then, there exists a unique solution~$\phi \in H^3(0,T) \cap H^2_{0,\bullet}(0,T) \cap W_\infty^3(0,T)$ satisfying
\begin{align}
     \label{eq:33} \| \partial_t^3 \phi \|_{L^2(0,T)} & \le C_\mu \| \partial_t f\|_{L^2(0,T)}, \quad \| \partial_t^3 \phi \|_{L^\infty(0,T)} \le C_\mu \| \partial_t f\|_{L^2(0,T)},
\end{align}
with a constant~$C_\mu > 0$ depending only on~$T$ and $\mu$.
\end{lemma}
\begin{proof}
Due to the coercivity property~\eqref{eq:15}, if a solution~$\phi\in H_{0,\bullet}^2(0,T)$ of problem~\eqref{eq:32} exists, it is unique. Let us define~$z := \L_T \phi$. Recalling the invertibility of $\L_T$ in $H_{0,\bullet}^2(0,T)$ (see Proposition \ref{prop:24}), problem~\eqref{eq:32} is then equivalent to finding~$z \in H^2_{0,\bullet}(0,T)$ such that
\begin{equation*}
    a_\mu(w,z) = (f, \partial_t \L_T w)_{L^2(0,T)} \quad \text{for all~} w \in H_{0,\bullet}^2(0,T),
\end{equation*}
or equivalently, using~$\widetilde{f}(t) := e^{-t/T} f(t)$ and \textit{2.} of Proposition~\ref{prop:24}, to finding~$z \in H^2_{0,\bullet}(0,T)$ such that
\begin{equation} \label{eq:34}
    (\partial_t z,\partial_t^2 w)_{L^2(0,T)} + \partial_t z(0) \partial_t w(0) + \mu (\partial_t z,w)_{L^2(0,T)} = (\widetilde{f}, \partial_t w)_{L^2(0,T)},
\end{equation}
for all $w \in H_{0,\bullet}^2(0,T)$. We formally apply integration by parts to both sides to rewrite the problem as
\begin{equation} \label{eq:35}
    (\partial_t^3 z,w)_{L^2(0,T)} - \partial_t^2 z(T) w(T)  + \partial_t z(T) \partial_t w(T) + \mu (\partial_t z,w)_{L^2(0,T)} = -(\partial_t \widetilde{f},  w)_{L^2(0,T)},
\end{equation}
where we have also used~$\widetilde{f}(T)=w(0)=0$. From~\eqref{eq:35}, if~$z \in H^3(0,T)$ is the solution of the ODE problem
\begin{equation*}
    \begin{cases}
    \partial_t^3 z(t) + \mu \partial_t z(t) = - \partial_t \widetilde{f}(t) & t \in (0,T), \\
    \partial_t^2 z(T) = \partial_t z(T) = z(0) = 0,
    \end{cases}
\end{equation*}
then $z$ is also a solution of~\eqref{eq:34}. This ODE problem can be solved explicitly by introducing $y := \partial_t z$, for which the problem becomes
\begin{equation*}
    \begin{cases}
    \partial_t^2 y(t) + \mu y(t) = - \partial_t \widetilde{f}(t) & t \in (0,T), \\
    \partial_t y(T) = y(T) = 0,
    \end{cases}
\end{equation*}
whose unique solution is given by
\begin{equation*}
    y(t) = \frac{1}{\sqrt{\mu}} \int_t^T \sin(\sqrt{\mu}(s-t)) 
    \partial_s \widetilde{f} (s) \, \dd s = -\int_t^T \cos(\sqrt{\mu}(s-t)) e^{-s/T} f (s) \, \dd s.
\end{equation*}
Then, using~$z(0)=0$, the unique solution~$z$ of~\eqref{eq:34}, is~$z(t)=\int_0^t y(s)\,\dd s$. We deduce that~$\phi = \L_T^{-1}z$ is the (unique) solution of~\eqref{eq:32}. More explicitly, from~\textit{3.} of Proposition~\ref{prop:24}, 
\begin{equation*}
    \phi(t) = \int_0^t e^{s/T} \partial_s z (s) \, \dd s = \int_0^t e^{s/T} y (s) \, \dd s.
\end{equation*}
We obtain~\eqref{eq:33} by computing
\begin{align*}
    \partial_t^3 \phi(t) & = \frac{1}{T^2} e^{t/T} y(t) + \frac{2}{T} e^{t/T} \partial_t y(t) +  e^{t/T} \partial_t^2 y(t),
\end{align*}
substituting the expression of~$\partial_t y$ and~$\partial_t^2 y$ given by
\begin{align*}
    \partial_t y(t) & =e^{-t/T}f(t)-\sqrt{\mu}\int_t^T\sin(\sqrt{\mu}(s-t))e^{-s/T}f(s)\,\dd s,
    \\ \partial_t^2 y(t) & =-\mu y(t)- \partial_t \widetilde{f}(t) =-\mu y(t) + \frac{1}{T} e^{-t/T} f(t) - e^{-t/T} \partial_t f(t),
\end{align*}
and the expression of~$y$, and applying the Cauchy-Schwarz inequality, taking also into account that~$f(t)=-\int_t^T \partial_s f(s)\,\dd s$.
\end{proof}
\begin{remark} \label{rem:37}
The result of Lemma \ref{lem:36} is valid also when $\mu=0$, namely, given $f \in H^1(0,T)$ with $f(T)=0$, the problem
\begin{equation*}
    \text{find~} \phi \in H^2_{0,\bullet}(0,T) \text{~such that} \quad a_0(w,\L_T \phi) = (f,\partial_t \L_T w)_{L^2(0,T)} \quad \text{for all~} w \in H^2_{0,\bullet}(0,T)
\end{equation*}
has a unique solution $\phi \in H^3(0,T) \cap H^2_{0,\bullet}(0,T) \cap W^3_\infty(0,T)$ satisfying
\begin{align*}
    \| \partial_t^3 \phi \|_{L^2(0,T)} & \le C \| \partial_t f\|_{L^2(0,T)}, \quad \| \partial_t^3 \phi \|_{L^\infty(0,T)} \le C \| \partial_t f\|_{L^2(0,T)},
\end{align*}
with a constant~$C > 0$ depending only on~$T$.
\end{remark}

\noindent
In the following corollary, we derive error estimates in the~$L^2$ norm.
\begin{corollary} \label{cor:38}
Under the same assumptions as in Theorem~\ref{theo:34},
we have 
\begin{align} \label{eq:36}
    \| u - u_h \|_{L^2(0,T)} &\le C_\mu h^s \| \partial_t^{s+1} u \|_{L^2(0,T)},
\end{align}
where~$C_\mu$ is a positive constant depending on~$p, T, \mu$ but independent of~$h$.
\end{corollary}
\begin{proof}
In order to prove~\eqref{eq:36}, we consider the adjoint problem~\eqref{eq:32} with the source term
\begin{equation*}
    f(t) := e^{t/T} \int_t^T (u-u_h)(s) \, \dd s.
\end{equation*}
Since~$f \in H^1(0,T)$ and~$f(T)=0$, by Lemma~\ref{lem:36}, the problem is well-posed, and its unique solution~$\phi \in H^3(0,T) \cap H_{0,\bullet}^2(0,T) \cap W_\infty^3(0,T)$ satisfies the stability estimates in~\eqref{eq:33}. From~$\partial_t f(t) = T^{-1} f(t) - e^{t/T} (u-u_h)(t)$, the estimates in~\eqref{eq:33} become
\begin{equation} \label{eq:37}
\begin{aligned}
     \| \partial_t^3 \phi \|_{L^2(0,T)} & \le C_\mu \| u - u_h \|_{L^2(0,T)},
    \qquad \| \partial_t^3 \phi \|_{L^\infty(0,T)} & \le C_\mu \| u-u_h \|_{L^2(0,T)}.
\end{aligned}
\end{equation}
Moreover, using~\textit{2.} of Proposition~\eqref{prop:24}, 
$f(T)=0$ and~$(u-u_h)(0)=0$,
we compute
\begin{equation*}
    (f, \partial_t \L_T (u-u_h))_{L^2(0,T)} 
    = \| u - u_h \|^2_{L^2(0,T)}.
\end{equation*}
Therefore, using the latter relation, the auxiliary problem~\eqref{eq:32}, and employing consistency, the continuity in~\eqref{eq:17}, as well as estimate 
\eqref{eq:20}, we deduce, with $\widetilde{\phi}_h$ as in Assumption \ref{ass:32},
\begin{align*}
    \| u - u_ h \|_{L^2(0,T)}^2 & = (f, \partial_t \L_T( u -   u_h))_{L^2(0,T)} 
    \\ & = a_{\mu}(u-u_h, \L_T \phi)
    \\ & =  a_{\mu}(u - u_h, \L_T(\phi - \widetilde{\phi}_h))
    \\ & \le C_\mu \| u - u_h \|_{H_{0,\bullet}^2(0,T)} \left( \| \partial_t \phi - \partial_t \widetilde{\phi}_h \|_{L^2(0,T)} + | \partial_t \phi(0) - \partial_t \widetilde{\phi}_h(0)|\right).
\end{align*}
From this, we obtain~\eqref{eq:36} by employing the approximation estimates in~\eqref{eq:25} and~\eqref{eq:26} (with $s=2$), the stability estimates in~\eqref{eq:37}, and Corollary~\ref{cor:35}. 
\end{proof}

\subsection{Quasi-optimal convergence rates} \label{sec:32}

The convergence rate in~\eqref{eq:28} established solely under Assumption~\ref{ass:32} on the discrete spaces is suboptimal by one order, as are the rates in the~$H^2$ and~$L^2$ norms from Corollary~\ref{cor:35} and Corollary~\ref{cor:38}. For general spaces of splines, this result is sharp. To demonstrate this, we apply the numerical scheme~\eqref{eq:21} using $C^1$-continuous splines on a uniform mesh and vary the polynomial degree. The test data are given by
\begin{equation} \label{eq:38}
    T = 5, \quad u(t) = t^2 e^{-t}, \quad \text{and} \quad \mu = 10^5.
\end{equation}
In Figure~\ref{fig:1}, we present the relative errors in the~$H^2$,~$H^1$ and~$L^2$ norms. We observe that the estimates obtained in Section \ref{sec:31} are sharp when~$p$ is odd. However, for even values of~$p$, the numerical convergence rates are found to be quasi-optimal.

\begin{figure}[h!]
\centering
\begin{minipage}{0.33\textwidth}
\begin{tikzpicture}
\begin{groupplot}[group style={group size=1 by 1},height=6cm,width=5.7cm, every axis label={font=\normalsize}, ylabel style={font=\footnotesize}]
        
    \nextgroupplot[xmode=log, 
                log x ticks with fixed point,
                xtick={0.05,0.1,0.2,0.4},
                title =~$H^2$-error,
                ytick={0.1,0.001,0.00001,0.0000001,0.000000001},
                xlabel={$h$},
                ymode=log,
                legend pos=south west, 
                legend style={nodes={scale=1, transform shape}},]
    \pgfplotstableread{
        x       p2              p3              p4              p5              p6
        0.625      0.362322893791104   0.145649442729327   0.008804509120339   0.001459495746554   0.000034532319166
        0.3125     0.140130122395653   0.066755542175366   0.001000857285800   0.000174322874695   0.000001038031848
        0.15625    0.058184656595298   0.031924528851351   0.000115892963181   0.000021063302973   0.000000031878582
        0.078125     0.026183443226848   0.015366178808486   0.000014090277717   0.000002412301376   0.000000001003107
    } \HtwoErrors
    \addplot[orange, mark=*,mark size=3, line width=0.02cm] table[x=x, y=p2] \HtwoErrors;
    \addplot[cyan, mark=triangle*,mark size=3, line width=0.02cm] table[x=x, y=p3] \HtwoErrors;
    \addplot[magenta, mark=pentagon*,mark size=3, line width=0.02cm] table[x=x, y=p4] \HtwoErrors;
    \addplot[blue, mark=square*, mark size=3,line width=0.02cm] table[x=x, y=p5] \HtwoErrors;
    \addplot[darkgreen, mark=diamond*,mark size=3, line width=0.02cm] table[x=x, y=p6] \HtwoErrors;
    \logLogSlopeTriangleDual{0.37}{0.25}{0.725}{1}{orange}{cyan};
    \logLogSlopeTriangleDual{0.35}{0.25}{0.345}{3}{magenta}{blue};
    \logLogSlopeTriangle{0.31}{0.155}{0.09}{5}{darkgreen};
\end{groupplot}
\end{tikzpicture}
\end{minipage}
\begin{minipage}{0.33\textwidth}
\begin{tikzpicture}
\begin{groupplot}[group style={group size=1 by 1},height=6cm,width=5.7cm, every axis label={font=\normalsize}, ylabel style={font=\footnotesize}]
        
    \nextgroupplot[xmode=log, 
                log x ticks with fixed point,
                xtick={0.05,0.1,0.2,0.4},
                ytick={0.1,0.001,0.00001,0.0000001,0.000000001,0.00000000001},
                title =~$H^1$-error,
                xlabel={$h$},
                ymode=log,
                legend pos=south west, 
                legend style={nodes={scale=1, transform shape}},]
    \pgfplotstableread{
        x       p2              p3              p4              p5              p6
        0.625      0.207666245121251   0.057270067471367   0.002257894125842   0.000279155412738   0.000004805144825
        0.3125     0.048727875280743   0.015356114211205   0.000139188625994   0.000017747115971   0.000000074996771
        0.15625    0.011062792695288   0.003989595015838   0.000008386560069   0.000001104272028   0.000000001171755
        0.078125     0.002584030654227   0.001001820982063   0.000000520812440   0.000000064112432   0.000000000019155
    } \HoneErrors
    \addplot[orange, mark=*, mark size=3, line width=0.02cm] table[x=x, y=p2] \HoneErrors;
    \addplot[cyan, mark=triangle*, mark size=3, line width=0.02cm] table[x=x, y=p3] \HoneErrors;
    \addplot[magenta, mark=pentagon*, mark size=3, line width=0.02cm] table[x=x, y=p4] \HoneErrors;
    \addplot[blue, mark=square*, mark size=3, line width=0.02cm] table[x=x, y=p5] \HoneErrors;
    \addplot[darkgreen, mark=diamond*,  mark size=3,line width=0.02cm] table[x=x, y=p6] \HoneErrors;
    \logLogSlopeTriangleDual{0.35}{0.25}{0.675}{2}{orange}{cyan};
    \logLogSlopeTriangleDual{0.35}{0.25}{0.335}{4}{magenta}{blue};
    \logLogSlopeTriangle{0.295}{0.135}{0.09}{6}{darkgreen};
\end{groupplot}
\end{tikzpicture}
\end{minipage}
\vspace{0.1cm}
\begin{minipage}{0.32\textwidth}
\begin{tikzpicture}
\begin{groupplot}[group style={group size=1 by 1},height=6cm,width=5.7cm, every axis label={font=\normalsize}, ylabel style={font=\footnotesize}]
        
    \nextgroupplot[xmode=log, 
                log x ticks with fixed point,
                xtick={0.05,0.1,0.2,0.4},
                ytick={0.1,0.001,0.00001,0.0000001,0.000000001,0.00000000001,0.0000000000001},
                title =~$L^2$-error,
                xlabel={$h$},
                ymode=log,
                legend pos=south west, 
                legend style={nodes={scale=1, transform shape}},]
    \pgfplotstableread{
        x       p2              p3              p4              p5              p6
   0.625    0.018967002467416   0.002798478782804   0.000071268202215   0.000006381587276   0.000000084702421
   0.3125    0.002356339462548   0.000398864475533   0.000002297416327   0.000000214237761   0.000000000690874
   0.15625    0.000273194941792   0.000053533664641   0.000000070845952   0.000000006874670   0.000000000005527
   0.078125    0.000032083302760   0.000006840211472   0.000000002231036   0.000000000203121   0.000000000000048
    } \HoneErrors
    \addplot[orange, mark=*, mark size=3, line width=0.02cm] table[x=x, y=p2] \HoneErrors;
    \addplot[cyan, mark=triangle*, mark size=3, line width=0.02cm] table[x=x, y=p3] \HoneErrors;
    \addplot[magenta, mark=pentagon*, mark size=3, line width=0.02cm] table[x=x, y=p4] \HoneErrors;
    \addplot[blue, mark=square*, mark size=3, line width=0.02cm] table[x=x, y=p5] \HoneErrors;
    \addplot[darkgreen, mark=diamond*,  mark size=3,line width=0.02cm] table[x=x, y=p6] \HoneErrors;
    \logLogSlopeTriangleDual{0.35}{0.25}{0.615}{3}{orange}{cyan};
    \logLogSlopeTriangleDual{0.33}{0.25}{0.285}{5}{magenta}{blue};
    \logLogSlopeTriangle{0.29}{0.135}{0.09}{7}{darkgreen};
\end{groupplot}
\end{tikzpicture}
\end{minipage}
\caption{Relative errors in the~$H^2$ norm (left plot),~$H^1$ norm (center plot) and~$L^2$ norm (right plot) with $C^1$-splines for~$p=2$ ({\large{\textcolor{orange}{$\bullet$}}} marker),~$p=3$ ({\footnotesize{\textcolor{cyan}{$\blacktriangle$}}} marker),~$p=4$ ({\footnotesize{\textcolor{magenta}{$\pentagofill$}}} marker),~$p=5$ ({\footnotesize{\textcolor{blue}{\(\blacksquare\)}}} marker) and~$p=6$ ({\footnotesize{\textcolor{darkgreen}{\(\blacklozenge\)}}} marker) solving problem~\eqref{eq:21} with data as in~\eqref{eq:38}.}
\label{fig:1}
\end{figure}
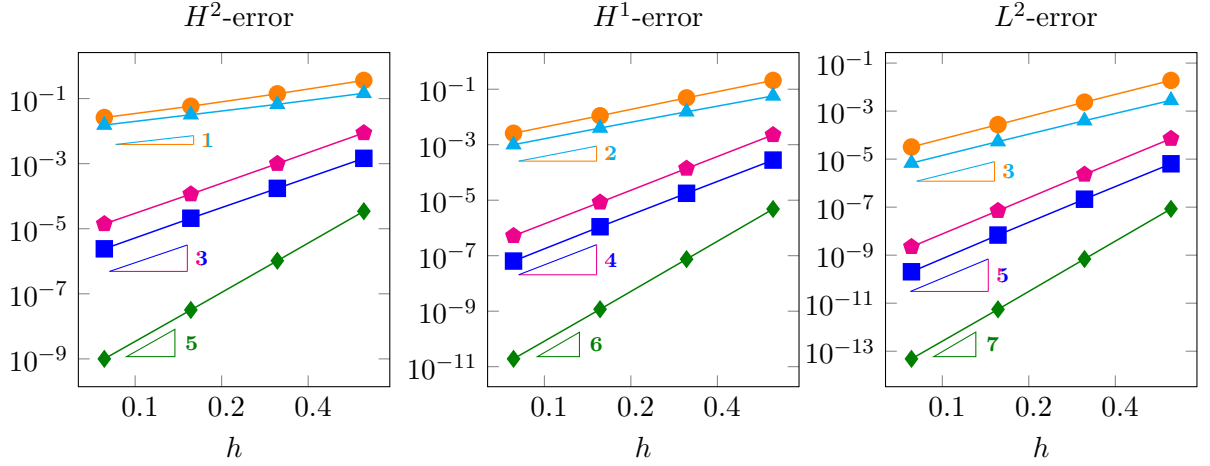

\noindent 
A different scenario arises when using splines of maximal regularity. In Figure~\ref{fig:2}, we report the errors
for various degrees $p$ (and thus regularities $p-1$) for the same test problem as in~\eqref{eq:38}. In this case, the obtained convergence rates are quasi-optimal both for even and odd degrees.

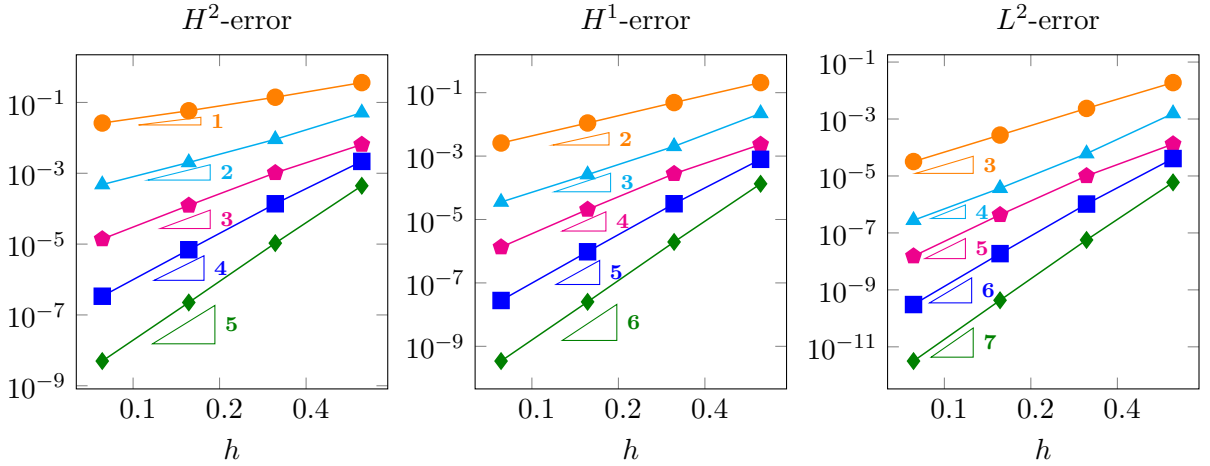
\begin{figure}[h!]
\centering
\begin{minipage}{0.3315\textwidth}
\begin{tikzpicture}
\begin{groupplot}[group style={group size=1 by 1},height=6cm,width=5.7cm, every axis label={font=\normalsize}, ylabel style={font=\footnotesize}]
        
    \nextgroupplot[xmode=log, 
                log x ticks with fixed point,
                xtick={0.05,0.1,0.2,0.4},
                title =~$H^2$-error,
                ytick={0.1,0.001,0.00001,0.0000001,0.000000001},
                xlabel={$h$},
                ymode=log,
                legend pos=south west, 
                legend style={nodes={scale=1, transform shape}},]
    \pgfplotstableread{
        x       p2              p3              p4              p5              p6
        0.625      0.362322893791104   0.051155155352790   0.006459485006700   0.002162254244325   0.000443169147154
        0.3125     0.140130122395653   0.009074258191251   0.001026845350395   0.000137601051777   0.000010647290719
        0.15625    0.058184656595298   0.002015589516520   0.000123383678505   0.000006934226995   0.000000224587909
        0.078125     0.026183443226848   0.000478589249619   0.000013889442878   0.000000337468625   0.000000004983848
    } \HtwoErrors
    \addplot[orange, mark=*,mark size=3, line width=0.02cm] table[x=x, y=p2] \HtwoErrors;
    \addplot[cyan, mark=triangle*,mark size=3, line width=0.02cm] table[x=x, y=p3] \HtwoErrors;
    \addplot[magenta, mark=pentagon*,mark size=3, line width=0.02cm] table[x=x, y=p4] \HtwoErrors;
    \addplot[blue, mark=square*, mark size=3,line width=0.02cm] table[x=x, y=p5] \HtwoErrors;
    \addplot[darkgreen, mark=diamond*,mark size=3, line width=0.02cm] table[x=x, y=p6] \HtwoErrors;
    \logLogSlopeTriangle{0.4}{0.2}{0.79}{1}{orange};
    \logLogSlopeTriangle{0.43}{0.2}{0.625}{2}{cyan};    
    \logLogSlopeTriangle{0.43}{0.16}{0.48}{3}{magenta};
    \logLogSlopeTriangle{0.41}{0.16}{0.325}{4}{blue};
    \logLogSlopeTriangle{0.445}{0.2}{0.135}{5}{darkgreen};
\end{groupplot}
\end{tikzpicture}
\end{minipage}
\begin{minipage}{0.3335\textwidth}
\begin{tikzpicture}
\begin{groupplot}[group style={group size=1 by 1},height=6cm,width=5.7cm, every axis label={font=\normalsize}, ylabel style={font=\footnotesize}]
        
    \nextgroupplot[xmode=log, 
                log x ticks with fixed point,
                xtick={0.05,0.1,0.2,0.4},
                ytick={0.1,0.001,0.00001,0.0000001,0.000000001,0.00000000001},
                title =~$H^1$-error,
                xlabel={$h$},
                ymode=log,
                legend pos=south west, 
                legend style={nodes={scale=1, transform shape}},]
    \pgfplotstableread{
        x       p2              p3              p4              p5              p6
        0.625      0.207666245121251   0.022169499590134   0.002318324051685   0.000790905989190   0.000133736480337
        0.3125     0.048727875280743   0.002024228765156   0.000276969527331   0.000031404754998   0.000001956674949
        0.15625    0.011062792695288   0.000260313618456   0.000020659659408   0.000000959632278   0.000000025259204
        0.078125     0.002584030654227   0.000035503081048   0.000001344452117   0.000000027890204   0.000000000342531
    } \HoneErrors
    \addplot[orange, mark=*, mark size=3, line width=0.02cm] table[x=x, y=p2] \HoneErrors;
    \addplot[cyan, mark=triangle*, mark size=3, line width=0.02cm] table[x=x, y=p3] \HoneErrors;
    \addplot[magenta, mark=pentagon*, mark size=3, line width=0.02cm] table[x=x, y=p4] \HoneErrors;
    \addplot[blue, mark=square*, mark size=3, line width=0.02cm] table[x=x, y=p5] \HoneErrors;
    \addplot[darkgreen, mark=diamond*,  mark size=3,line width=0.02cm] table[x=x, y=p6] \HoneErrors;
    \logLogSlopeTriangle{0.43}{0.18}{0.73}{2}{orange};    
    \logLogSlopeTriangle{0.435}{0.18}{0.59}{3}{cyan};    
    \logLogSlopeTriangle{0.42}{0.14}{0.4725}{4}{magenta};
    \logLogSlopeTriangle{0.4}{0.14}{0.3125}{5}{blue};
    \logLogSlopeTriangle{0.455}{0.175}{0.145}{6}{darkgreen};
\end{groupplot}
\end{tikzpicture}
\end{minipage}
\vspace{0.1cm}
\begin{minipage}{0.32\textwidth}
\begin{tikzpicture}
\begin{groupplot}[group style={group size=1 by 1},height=6cm,width=5.7cm, every axis label={font=\normalsize}, ylabel style={font=\footnotesize}]
        
    \nextgroupplot[xmode=log, 
                log x ticks with fixed point,
                xtick={0.05,0.1,0.2,0.4},
                ytick={0.1,0.001,0.00001,0.0000001,0.000000001,0.00000000001,0.0000000000001},
                title =~$L^2$-error,
                xlabel={$h$},
                ymode=log,
                legend pos=south west, 
                legend style={nodes={scale=1, transform shape}},]
    \pgfplotstableread{
        x       p2              p3              p4              p5              p6
   0.625    0.018967002467416   0.001559435568554   0.000131490817258   0.000040898062055   0.000005962854515
   0.3125    0.002356339462548   0.000060863347012   0.000010045746156   0.000001033637063   0.000000056641102
   0.15625    0.000273194941792   0.000003695048605   0.000000432044574   0.000000018748336   0.000000000431368
   0.078125    0.000032083302760   0.000000278092041   0.000000015304050   0.000000000305623   0.000000000003165
    } \HoneErrors
    \addplot[orange, mark=*, mark size=3, line width=0.02cm] table[x=x, y=p2] \HoneErrors;
    \addplot[cyan, mark=triangle*, mark size=3, line width=0.02cm] table[x=x, y=p3] \HoneErrors;
    \addplot[magenta, mark=pentagon*, mark size=3, line width=0.02cm] table[x=x, y=p4] \HoneErrors;
    \addplot[blue, mark=square*, mark size=3, line width=0.02cm] table[x=x, y=p5] \HoneErrors;
    \addplot[darkgreen, mark=diamond*,  mark size=3,line width=0.02cm] table[x=x, y=p6] \HoneErrors;
    \logLogSlopeTriangle{0.275}{0.185}{0.645}{3}{orange};
    \logLogSlopeTriangle{0.25}{0.11}{0.51}{4}{cyan};
    \logLogSlopeTriangle{0.25}{0.13}{0.39}{5}{magenta};
    \logLogSlopeTriangle{0.27}{0.135}{0.2575}{6}{blue};
    \logLogSlopeTriangle{0.275}{0.135}{0.095}{7}{darkgreen};
\end{groupplot}
\end{tikzpicture}
\end{minipage}
\caption{Relative errors in the~$H^2$ norm (left plot),~$H^1$ norm (center plot) and~$L^2$ norm (right plot) with maximal regularity splines for~$p=2$ ({\large{\textcolor{orange}{$\bullet$}}} marker),~$p=3$ ({\footnotesize{\textcolor{cyan}{$\blacktriangle$}}} marker),~$p=4$ ({\footnotesize{\textcolor{magenta}{$\pentagofill$}}} marker),~$p=5$ ({\footnotesize{\textcolor{blue}{\(\blacksquare\)}}} marker) and~$p=6$ ({\footnotesize{\textcolor{darkgreen}{\(\blacklozenge\)}}} marker) solving problem~\eqref{eq:21} with data as 
in~\eqref{eq:38}.}
\label{fig:2}
\end{figure}

\noindent
Based on numerical findings, our conjecture is that, for splines on uniform meshes, quasi-optimal convergence occurs whenever the difference between polynomial degree and regularity is odd.
\begin{remark}
On a mesh with $N$ elements, the dimension of~$S^{q,q-1}(0,T)$ is~$N+q$. Therefore, for maximal regularity splines, even if the convergence rate were suboptimal by one order, achieving a convergence rate of~$p$ for the~$H^1$ error by switching from~$S^{p,p-1}(0,T)$ to $S^{p+1,p}(0,T)$ would increase the number of degrees of freedom by just one.
\end{remark}
\noindent
In this section, we prove a quasi-optimal convergence result, under an additional assumption on the discrete spaces~$S_h^p(0,T) \subset H_{0,\bullet}^2(0,T)$. Then, in Section~\ref{sec:33}, we prove that, for $C^1$-continuous splines of even degree on uniform meshes, this assumption is actually satisfied.
\begin{assumption} \label{ass:38}
There exists a projection operator~$Q_h : H^2_{0,\bullet}(0,T) \to S_h^p(0,T)$ such that
\begin{equation*}
\begin{cases}
    \partial_t Q_h u(T) = \partial_t u(T) \\
    (\partial_t Q_h u, e^{-\bullet/T} \partial_t^2 w_h )_{L^2(0,T)} = (\partial_t u, e^{-\bullet/T} \partial_t^2 w_h )_{L^2(0,T)}
    \quad\quad \text{for all~} w_h \in S_h^p(0,T),
\end{cases}
\end{equation*}
and such that, for every $u$ sufficiently smooth,
\begin{equation} \label{eq:39}
    \| \partial_t u - \partial_t Q_h u \|_{L^2(0,T)} \le C h^p,
\end{equation}
with a constant~$C>0$ depending on $u$ but independent of~$h$ .
\end{assumption}
\noindent
Consider also the auxiliary projection operator $\Pi_h^{\partial_t^2} : H_{0,\bullet}^2(0,T) \to S_h^p(0,T)$ defined by
\begin{equation} \label{eq:40}
    (\partial_t^2 \Pi_h^{\partial_t^2} u, \partial_t \L_T w_h )_{L^2(0,T)} + \partial_t \Pi_h^{\partial_t^2} u(0) \partial_t w_h(0) = (\partial_t^2 u, \partial_t \L_T w_h )_{L^2(0,T)} + \partial_t u(0) \partial_t w_h(0)
\end{equation}
for all $w_h \in S_h^p(0,T)$. Note that~\eqref{eq:40} can also be written as~$a_0(\Pi_h^{\partial_t^2} u, \L_T w_h)=a_0(u, \L_T w_h)$, with $a_0$ as in \eqref{eq:5} with $\mu=0$.
\begin{proposition} \label{prop:310n}
The projection operator $\Pi_h^{\partial_t^2}$ defined in \eqref{eq:40} is well-defined and satisfies
\begin{align}
\label{eq:41}
    \| \partial_t u - \partial_t \Pi_h^{\partial_t^2} u\|_{L^2(0,T)} & \le (1+2e) \| \partial_t u - \partial_t Q_h u\|_{L^2(0,T)}, 
\end{align}
with $Q_h$ defined in Assumption~\ref{ass:38}.
\end{proposition}
\begin{proof}
With \textit{4.} and \textit{6.} of Proposition \ref{prop:24}, we deduce the coercivity property
\begin{equation} \label{eq:42}
    (\partial_t^2 w, \partial_t \L_T w)_{L^2(0,T)} + |\partial_t w(0)|^2 \ge \frac{1}{2eT} \|\partial_t w\|^2_{L^2(0,T)} \quad \text{for all~} w \in H^2_{0,\bullet}(0,T).
\end{equation}
This implies that the square, finite dimensional system~\eqref{eq:40} is nonsingular, and therefore~$\Pi_h^{\partial_t^2}$ is well defined. To obtain estimate~\eqref{eq:41}, we first write, with the triangle inequality,
\begin{align} \label{eq:43}
    \| \partial_t u - \partial_t \Pi_h^{\partial_t^2} u \|_{L^2(0,T)} \le \| \partial_t u - \partial_t Q_h u \|_{L^2(0,T)} + \| \partial_t \Pi_h^{\partial_t^2} u - \partial_t Q_h u \|_{L^2(0,T)}.
\end{align}
Then, we use \eqref{eq:42}, the definition of $\Pi_h^{\partial_t^2}$ in \eqref{eq:40}, property~\emph{2.} of Proposition~\ref{prop:24} together with integration by parts, and the definition of $Q_h$ in Assumption \ref{ass:38},  to deduce
\begin{align*}
    \| \partial_t \Pi_h^{\partial_t^2}  u - \partial_t Q_h u\|^2_{L^2(0,T)} & \le 2eT \Bigl(\partial_t^2 ( \Pi_h^{\partial_t^2} u - Q_h u), \partial_t \L_T (\Pi_h^{\partial_t^2} u - Q_h u))_{L^2(0,T)} 
    \\ & \qquad\qquad + \partial_t (\Pi_h^{\partial_t^2} u - Q_h u)(0)\, \partial_t (\Pi_h^{\partial_t^2}u - Q_h u)(0) \Bigr)
    \\ & \le 2eT \Bigl(\partial_t^2 (u - Q_h u), \partial_t \L_T (\Pi_h^{\partial_t^2} u - Q_h u))_{L^2(0,T)} 
    \\ & \qquad\qquad + \partial_t (u - Q_h u)(0) \,\partial_t (\Pi_h^{\partial_t^2} u - Q_h u)(0) \Bigr)
    \\ & = 2e (\partial_t (u - Q_h u), \partial_t \L_T ( \Pi_h^{\partial_t^2} u - Q_h u))_{L^2(0,T)}.
\end{align*}
With Cauchy-Schwarz and \textit{8.} of Proposition \ref{prop:24} we conclude 
\begin{equation} \label{eq:44}
    \| \partial_t \Pi_h^{\partial_t^2} u - \partial_t Q_h u\|_{L^2(0,T)} \le 2e \| \partial_t u - \partial_t Q_h u \|_{L^2(0,T)}.
\end{equation}
Combining \eqref{eq:43} and \eqref{eq:44}, we deduce \eqref{eq:41}.
\end{proof}
\noindent
The following approximation results easily follows.
\begin{corollary} \label{cor:312}
Under Assumptions~\ref{ass:32} and~\ref{ass:38}, assuming $u$ sufficiently smooth, the projection operator $\Pi_h^{\partial_t^2}$ defined in \eqref{eq:40} satisfies
\begin{align}
\label{eq:45}
    \| \partial_t^\ell u - \partial_t^\ell \Pi_h^{\partial_t^2} u\|_{L^2(0,T)} & \le C h^{p+1-\ell},\qquad \ell=0,1,2,  
\end{align}
with a constant $C>0$ depending on~$u$ but independent of~$h$.
\end{corollary}
\begin{proof}
For~$\ell=1$, estimate~\eqref{eq:45} immediately follows from Proposition~\ref{prop:310n} and estimate~\eqref{eq:39}. For~$\ell=2$, the results is obtained by proceeding as in Corollary~\ref{cor:35} and using~\eqref{eq:45} for~$\ell=1$. Finally, the result for~$\ell=0$ is obtained as in Corollary~\ref{cor:38}, with~$u-\Pi_h^{\partial_t^2} u$ instead of~$u-u_h$, and~$a_0(\cdot,\cdot)$ instead of~$a_\mu(\cdot,\cdot)$, owing to Remark~\ref{rem:37}, and using~\eqref{eq:45} for~$\ell=2$.
\end{proof}
\noindent
We also prove the following auxiliary result.
\begin{lemma} \label{lem:313}
Let~$u \in H^2_{0,\bullet}(0,T)$ be the unique solution of problem~\eqref{eq:3}, and let~$u_h \in S_h^p(0,T)$ be the unique solution of problem~\eqref{eq:21}. Then, it holds
\begin{equation*}
    \| \partial_t u_h - \partial_t \Pi_h^{\partial_t^2} u  \|_{L^2(0,T)} \le 2 e T \mu \| u-\Pi_h^{\partial_t^2} u  \|_{L^2(0,T)},
\end{equation*}
with the projector $\Pi_h^{\partial_t^2}$ defined in \eqref{eq:40}.
\end{lemma}
\begin{proof}
From the coercivity in \eqref{eq:15}, consistency, and the definition of $\Pi_h^{\partial_t^2}$, we deduce
\begin{align*}
    \| \partial_t u_h - \partial_t \Pi_h^{\partial_t^2} u \|_{L^2(0,T)}^2 & \le 2eT \,a_\mu( u_h  - \Pi_h^{\partial_t^2} u , \L_T(u_h - \Pi_h^{\partial_t^2} u ))
     \\ & = 2eT\, a_\mu( u -\Pi_h^{\partial_t^2} u , \L_T(u_h - \Pi_h^{\partial_t^2} u))
    \\ & = 2eT\mu (u-\Pi_h^{\partial_t^2} u , \partial_t \L_T (u_h-\Pi_h^{\partial_t^2} u ))_{L^2(0,T)}
    \\ & \le 2eT\mu \| u - \Pi_h^{\partial_t^2} u \|_{L^2(0,T)}  \| \partial_t u_h - \partial_t  \Pi_h^{\partial_t^2} u \|_{L^2(0,T)},
\end{align*}
where in the last step we used the Cauchy-Schwarz inequality and~\textit{8.} of Proposition~\ref{prop:24}. The result immediately follows.
\end{proof}

\noindent
In the next theorem, we prove that Assumption \ref{ass:38} is sufficient to guarantee a quasi-optimal convergence result.
\begin{theorem} \label{th:314}
Let the discrete spaces~$S_h^p(0,T) \subset H^2_{0,\bullet}(0,T)$ satisfy Assumptions~\ref{ass:32} and~\ref{ass:38}. Let~$u \in H^2_{0,\bullet}(0,T)$ be the unique solution of problem~\eqref{eq:3}, and let~$u_h \in S_h^p(0,T)$ be the unique solution of problem~\eqref{eq:21}. Then, if~$u$ is sufficiently smooth, we have
\begin{align} \label{eq:46}
    \| \partial_t^\ell u - \partial_t^\ell u_h \|_{L^2(0,T)} &\le C_\mu h^{p+1-\ell}, \quad \ell=0,1,2
\end{align}
for a positive constant~$C_\mu$ depending on~$u, p, T, \mu$ but not on~$h$.
\end{theorem}
\begin{proof}
With the triangle inequality, Lemma~\ref{lem:313}, and Corollary \ref{cor:312}, we obtain
\begin{align*}
     \| \partial_t u - \partial_t u_h \|_{L^2(0,T)} & \le \| \partial_t u - \partial_t \Pi_h^{\partial_t^2} u \|_{L^2(0,T)} + \| \partial_t u_h - \partial_t \Pi_h^{\partial_t^2} u \|_{L^2(0,T)}
     \\ & \le \| \partial_t u - \partial_t \Pi_h^{\partial_t^2} u \|_{L^2(0,T)} + 2 e T \mu\| u - \Pi_h^{\partial_t^2} u \|_{L^2(0,T)}
     \\ & \le C h^p,
\end{align*}
which proves \eqref{eq:46} for $\ell=1$. For~$\ell=2$, estimate~\eqref{eq:46} is obtained exactly as in Corollary~\ref{cor:35}. Finally, for $\ell=0$, we proceed similarly as for $\ell=1$, without resorting to a duality argument.
With the triangle inequality, Lemma~\ref{lem:313}, the Poincar\'e inequality, and Corollary \ref{cor:312}, we obtain
\begin{align*}
     \| u - u_h \|_{L^2(0,T)} & \le \| u - \Pi_h^{\partial_t^2} u \|_{L^2(0,T)} + \| u_h - \Pi_h^{\partial_t^2} u \|_{L^2(0,T)}
     \\ & \le \| u - \Pi_h^{\partial_t^2} u \|_{L^2(0,T)} + \frac{2T}{\pi} \|\partial_t u_h - \partial_t \Pi_h^{\partial_t^2} u \|_{L^2(0,T)}
     \\ & \le \| u - \Pi_h^{\partial_t^2} u \|_{L^2(0,T)} + \frac{4eT^2\mu}{\pi} \|u -  \Pi_h^{\partial_t^2} u \|_{L^2(0,T)}
     \\ & \le C h^{p+1},
\end{align*}
and the proof is complete.
\end{proof}

\noindent
 In the next section, we prove that the spaces of $C^1$ splines of even degree on uniform meshes satisfy Assumption~\ref{ass:38}, thereby ensuring the quasi-optimal convergence of Theorem~\ref{th:314}. In Figure \ref{fig:3}, we provide numerical evidence suggesting that this may also applies to other situations. Specifically, for~$u(t) = \sin^2 t \, e^{-t} \, t^2$ and~$T = 5$, we compute the projection as described in Assumption~\ref{ass:38} for spline spaces with varying regularities. By varying the polynomial degree $p$, we report the $H^1$ error for regularities $p-1$, $p-2$, and $p-3$. As anticipated, these results suggest that Assumption \ref{ass:38} may be satisfied whenever the difference between the polynomial degree and the spline regularity is odd.

\begin{figure}[h!]
\centering
\begin{minipage}{0.3315\textwidth}
\begin{tikzpicture}
\begin{groupplot}[group style={group size=1 by 1},height=6cm,width=5.7cm, every axis label={font=\normalsize}, ylabel style={font=\footnotesize}]
        
    \nextgroupplot[xmode=log, 
                log x ticks with fixed point,
                xtick={0.05,0.1,0.2,0.4},
                title =Regularity $p-1$,
                ytick={0.1,0.001,0.00001,0.0000001,0.000000001},
                xlabel={$h$},
                ymode=log,
                legend pos=south west, 
                legend style={nodes={scale=1, transform shape}},]
    \pgfplotstableread{
        x       p2              p3              p4              p5              p6
        0.625         0.140407820604326   0.050587324638522   0.018139034564928   0.006787629940795   0.002222651351157
        0.3125        0.025516923114729   0.004069227070861   0.000718918396725   0.000150476357443   0.000019319863086
        0.15625       0.005878586776659   0.000407660488438   0.000044699537646   0.000003740567281   0.000000230448013
        0.078125        0.001436226992858   0.000048977742412   0.000002827211537   0.000000109060158   0.000000003189698
    } \HtwoErrors
    \addplot[orange, mark=*,mark size=3, line width=0.02cm] table[x=x, y=p2] \HtwoErrors;
    \addplot[cyan, mark=triangle*,mark size=3, line width=0.02cm] table[x=x, y=p3] \HtwoErrors;
    \addplot[magenta, mark=pentagon*,mark size=3, line width=0.02cm] table[x=x, y=p4] \HtwoErrors;
    \addplot[blue, mark=square*, mark size=3,line width=0.02cm] table[x=x, y=p5] \HtwoErrors;
    \addplot[darkgreen, mark=diamond*,mark size=3, line width=0.02cm] table[x=x, y=p6] \HtwoErrors;
    \logLogSlopeTriangle{0.295}{0.2}{0.665}{2}{orange};
    \logLogSlopeTriangle{0.26}{0.15}{0.53}{3}{cyan};    
    \logLogSlopeTriangle{0.25}{0.13}{0.4}{4}{magenta};
    \logLogSlopeTriangle{0.265}{0.13}{0.265}{5}{blue};
    \logLogSlopeTriangle{0.265}{0.14}{0.1}{6}{darkgreen};
\end{groupplot}
\end{tikzpicture}
\end{minipage}
\begin{minipage}{0.3335\textwidth}
\begin{tikzpicture}
\begin{groupplot}[group style={group size=1 by 1},height=6cm,width=5.7cm, every axis label={font=\normalsize}, ylabel style={font=\footnotesize}]
        
    \nextgroupplot[xmode=log, 
                log x ticks with fixed point,
                xtick={0.05,0.1,0.2,0.4},
                ytick={0.1,0.001,0.00001,0.0000001,0.000000001,0.00000000001},
                title =Regularity $p-2$,
                xlabel={$h$},
                ymode=log,
                legend pos=south west, 
                legend style={nodes={scale=1, transform shape}},]
    \pgfplotstableread{
        x       p2              p3              p4              p5              
        0.625         0.017942138832768   0.003870721159774   0.000896524458631   0.000244107650837                   
        0.3125        0.004463348073920   0.000459720853585   0.000073109258797   0.000010062936657                   
        0.15625       0.001102265949720   0.000056947690478   0.000004995929216   0.000000314084019                   
        0.078125        0.000274429755745   0.000007128963129   0.000000318214077   0.000000009548096                 
    } \HoneErrors
    \addplot[cyan, mark=triangle*, mark size=3, line width=0.02cm] table[x=x, y=p2] \HoneErrors;
    \addplot[magenta, mark=pentagon*, mark size=3, line width=0.02cm] table[x=x, y=p3] \HoneErrors;
    \addplot[blue, mark=square*, mark size=3, line width=0.02cm] table[x=x, y=p4] \HoneErrors;
    \addplot[darkgreen, mark=diamond*,  mark size=3,line width=0.02cm] table[x=x, y=p5] \HoneErrors;
    \logLogSlopeTriangle{0.415}{0.18}{0.68}{2}{cyan};    
    \logLogSlopeTriangle{0.385}{0.14}{0.5}{3}{magenta};
    \logLogSlopeTriangle{0.37}{0.14}{0.325}{4}{blue};
    \logLogSlopeTriangle{0.395}{0.175}{0.145}{5}{darkgreen};
\end{groupplot}
\end{tikzpicture}
\end{minipage}
\vspace{0.1cm}
\begin{minipage}{0.32\textwidth}
\begin{tikzpicture}
\begin{groupplot}[group style={group size=1 by 1},height=6cm,width=5.7cm, every axis label={font=\normalsize}, ylabel style={font=\footnotesize}]
        
    \nextgroupplot[xmode=log, 
                log x ticks with fixed point,
                xtick={0.05,0.1,0.2,0.4},
                ytick={0.1,0.001,0.00001,0.0000001,0.000000001,0.00000000001,0.0000000000001},
                title =Regularity $p-3$,
                xlabel={$h$},
                ymode=log,
                legend pos=south west, 
                legend style={nodes={scale=1, transform shape}},]
    \pgfplotstableread{
        x       p2              p3              p4              
   0.625       0.003957202184872   0.000413295464066   0.000070263356953                  
   0.3125       0.000180296500914   0.000011570973221   0.000000685871943
   0.15625       0.000010505902403   0.000000345709309   0.000000009296036            
   0.078125       0.000000646100208   0.000000010668927   0.000000000145772
    } \HoneErrors
    \addplot[magenta, mark=pentagon*, mark size=3, line width=0.02cm] table[x=x, y=p2] \HoneErrors;
    \addplot[blue, mark=square*, mark size=3, line width=0.02cm] table[x=x, y=p3] \HoneErrors;
    \addplot[darkgreen, mark=diamond*,  mark size=3,line width=0.02cm] table[x=x, y=p4] \HoneErrors;
    \logLogSlopeTriangle{0.25}{0.13}{0.475}{4}{magenta};
    \logLogSlopeTriangle{0.27}{0.135}{0.275}{5}{blue};
    \logLogSlopeTriangle{0.275}{0.135}{0.095}{6}{darkgreen};
\end{groupplot}
\end{tikzpicture}
\end{minipage}
\caption{Relative errors in the~$H^1$ norm with splines of regularity $p-1$ (left plot), $p-2$ (center plot), and $p-3$ (right plot), for~$p=2$ ({\large{\textcolor{orange}{$\bullet$}}} marker),~$p=3$ ({\footnotesize{\textcolor{cyan}{$\blacktriangle$}}} marker),~$p=4$ ({\footnotesize{\textcolor{magenta}{$\pentagofill$}}} marker),~$p=5$ ({\footnotesize{\textcolor{blue}{\(\blacksquare\)}}} marker), and~$p=6$ ({\footnotesize{\textcolor{darkgreen}{\(\blacklozenge\)}}} marker), computing the projection defined in Assumption~\ref{ass:38} for~$u(t) = \sin^2 t \, e^{-t} t^2$ and~$T=5$.}
\label{fig:3}
\end{figure}
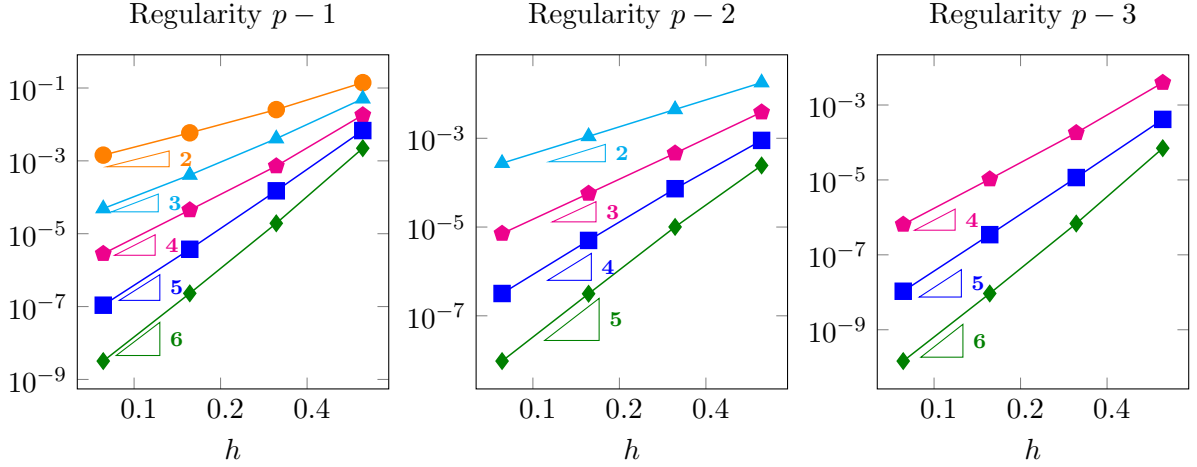

\subsection{Non-local projection into continuous piecewise polynomial spaces} \label{sec:33}

In this section, we prove that Assumption~\ref{ass:38} is satisfied for spaces of splines with $C^1$-regularity and even polynomial degree on uniform meshes. In order to do so, we first introduce and study a projection operator into spaces of continouous piecewise polynomial functions in  Propositions~\ref{prop:315} and~\ref{prop:316}, then we use it to characterize the operator~$Q_h$ in Assumption~\ref{ass:38} for $C^1$-splines of even degree~$p$, see Corollary~\ref{cor:317}.

\noindent
Consider the space $S_h^{p,0}(0,T)$ of continuous piecewise polynomials of degree $p$ defined on the uniform mesh $\{t_j = jh \mid j = 0,\ldots,N\}$ with mesh size $h = T/N$. Similarly, let $S_h^{p-1,-1}(0,T)$ be the space of discontinuous piecewise polynomials of degree $p-1$ over the same mesh. 
We set
\begin{equation} \label{eq:47}
    \widetilde{W}_\infty^{p+2}(0,T) := \{ v \in W_\infty^{p+2}(0,T) \mid \partial_t^{p+2} v(t) \text{~exists for all~} t \in (0,T)\}.
\end{equation}

\begin{proposition} \label{prop:315}
The projection operator~$\mathcal{P}_h^p : H^1(0,T) \to S_h^{p,0}(0,T)$ defined by
\begin{equation} \label{eq:48}
    \begin{cases}
         \mathcal{P}_h^p v (0) = v(0), &
        \\ (\mathcal{P}_h^p v, e^{-\bullet/T} q_h^{p-1})_{L^2(0,T)} = (v, e^{-\bullet/T} q_h^{p-1})_{L^2(0,T)} & \text{for all~} q_h^{p-1} \in S_h^{p-1,-1}(0,T),
    \end{cases}
\end{equation}
is well-defined. Moreover, if $v \in \widetilde{W}_\infty^{p+2}(0,T)$, where $\widetilde{W}_\infty^{p+2}(0,T)$ is defined in~\eqref{eq:47}, the following bound holds
\begin{equation} \label{eq:49}
    \max_{i=1,\ldots,N} | v(t_i) - \mathcal{P}_h^p v(t_i) | \le C h^{p+(p\hspace{-0.2cm}\mod{2})} \Bigl(\| \partial_t^{p+1} v \|_{L^\infty(0,T)} + \| \partial_t^{p+2} v \|_{L^\infty(0,T)}\Bigr),
\end{equation}
where the constant $C$ depends only on $p$ and $T$.
\end{proposition}
\begin{proof}
See Appendix~\ref{app:B1}.
\end{proof}
\noindent
The projection error in the $L^2$ norm satisfies the same estimate as in~\eqref{eq:49}.
\begin{proposition} \label{prop:316}
Let $v \in \widetilde{W}_\infty^{p+2}(0,T)$, where $\widetilde{W}_\infty^{p+2}(0,T)$ is defined in \eqref{eq:47}, and let~$\mathcal{P}_h^p$ be the projection operator defined in~\eqref{eq:48}. Then, the following bound holds
\begin{equation*}
    \| v - \mathcal{P}_h^p v \|_{L^2(0,T)} \le C h^{p+(p \hspace{-0.2cm} \mod 2)} \left(\| \partial_t^{p+1} v \|_{L^\infty(0,T)} +  \| \partial_t^{p+2} v \|_{L^\infty(0,T)} \right),
\end{equation*}
where the constant~$C$ depends only on~$p$ and~$T$.
\end{proposition}
\begin{proof}
See Appendix~\ref{app:B2}.
\end{proof}
\noindent 
We conclude this section by proving that~$C^1$-splines of even degree satisfy Assumption~\ref{ass:38}.
\begin{corollary} \label{cor:317}
Assumption~\ref{ass:38} is satisfied for $C^1$-continuous splines of degree $p \ge 2$, when~$p$ is even.
\end{corollary}
\begin{proof}
In~$S_h^{p,1}(0,T)$, we define the operator
\begin{equation*}
    u\,\mapsto \, Q_h u(t) := \int_0^t \mathcal{P}_h^{p-1} \partial_t u(s) \dd s,
\end{equation*}
with $\mathcal{P}_h^{p-1}$ defined in~\eqref{eq:89}. It is immediate to see that~$Q_h$ coincides with the operator defined in~Assumption~\ref{ass:38}. The bound~\eqref{eq:39} follows from Proposition \ref{prop:316}.
\end{proof}

\begin{remark}
For the projection operator~$Q_h$ onto spline spaces of regularity~$C^1$, Proposition~\ref{prop:316} implies the error estimate
\begin{equation*}
    \| \partial_t u - \partial_t Q_h u \|_{L^2(0,T)} \le C h^{p-1+(p \hspace{-0.2cm} \mod 2)} \left( \| \partial_t^{p+1} u\|_{L^\infty(0,T)} + \| \partial_t^{p+2} u\|_{L^\infty(0,T)} \right)\ \forall u \in \widetilde{W}_\infty^{p+2}(0,T),
\end{equation*}
with a constant~$C>0$ depending only on~$p$ and~$T$. The sharpness of this estimate can be seen from Figure~\ref{fig:3}, line with the {\large{\textcolor{orange}{$\bullet$}}} marker in the left plot, line with the {\footnotesize{\textcolor{cyan}{$\blacktriangle$}}} marker in the central plot, and line with the {\footnotesize{\textcolor{magenta}{$\pentagofill$}}} marker in the right plot, corresponding to $C^1$-splines of degree~$p=2$, $p=3$, and~$p=4$, respectively.
\end{remark}

\section{Variational formulation of the PDE problem} \label{sec:4}

In this section, we introduce a variational formulation the complete problem~\eqref{eq:1} that extends the formulation in~\eqref{eq:14} of the ODE problem, and prove its well-posedness.

\noindent
Let us consider first the following well-posed variational formulation of problem~\eqref{eq:1}: find 
\begin{equation} \label{eq:50}
    U \in L^2(0,T; H_0^1(\Omega)) \cap H^1(0,T; L^2(\Omega))
\end{equation}
such that $U_{|_{t=0}} = \partial_t U_{|_{t=0}} =0$ almost everywhere in~$\Omega$ and, for almost every $t \in (0,T)$,
\begin{equation} \label{eq:51}
    \langle \partial_t^2 U(\cdot,t), V \rangle_{H^1_0(\Omega)} + (c^2\nabla_{\bx} U(\cdot,t), \nabla_{\bx} V)_{L^2(\Omega)} = (F(\cdot,t), V)_{L^2(\Omega)}
\end{equation}
for all $V \in H_0^1(\Omega)$, where $\langle\cdot,\cdot\rangle_{H^1_0(\Omega)}$ is the duality product between~$[H_0^1(\Omega)]'$ and~$H^1_0(\Omega)$; see, e.g.,~\cite[Chapter 3. \S 8 Pag. 265]{LionsMagenes1972}. We restrict ourselves to a more regular framework and assume that
\begin{equation}\label{eq:52}
    F \in H^1(0,T;L^2(\Omega)), \quad F(\cdot,0) \in L^2(\Omega).
\end{equation}
We have the following result.
\begin{proposition}\label{prop:41}
Under the regularity assumption~\eqref{eq:52} on~$F$, there exists a unique solution to~\eqref{eq:51} and it satisfies
\begin{equation*}
    U \in H^1(0,T;H^1_0(\Omega))\cap H^2_{0,\bullet}(0,T;L^2(\Omega)).
\end{equation*}
\end{proposition}
\begin{proof}
See, e.g.,~\cite[Prop.~4.4]{AbdulleHenning2017} and \cite[Chapter~7.2]{EvansBook}. 
\end{proof}

\noindent
For functions $(U,W)$ such that
\begin{equation*}
    U\in L^2(0,T;H^1_0(\Omega))\cap H^2_{0,\bullet}(0,T;L^2(\Omega)),\quad W \in H^1(0,T;H^1_0(\Omega)) \ \text{ with }\  \partial_t W(\cdot, 0)\in L^2(\Omega),
\end{equation*}
we define the space--time bilinear form $\A$ as 
\begin{equation} \label{eq:53}
\begin{aligned}
    \A(U,W)  :=  (\partial_t^2 U, \partial_t W)_{L^2(Q_T)} + (\partial_t U(\cdot,0), \partial_t W(\cdot,0))_{L^2(\Omega)} +  (c^2\nabla_{\bx} U, \nabla_{\bx} \partial_t W)_{L^2(Q_T)}.
\end{aligned}
\end{equation}
We still denote by~$\L_T$ the operator analogue to that defined in~\eqref{eq:13} but now acting on space--time functions, i.e. we define $\L_T: H^2_{0,\bullet}(0,T;L^2(\Omega)) \to H^2_{0,\bullet}(0,T;L^2(\Omega))$ as
\begin{equation} \label{eq:54}
    \L_T W(\bx,t) : = \int_0^t e^{-s/T} \partial_s W(\bx,s) \, \dd s.
\end{equation}
Then, the space--time variational formulation of the complete wave problem~\eqref{eq:1} analogous to~\eqref{eq:14} reads as follows:

\begin{tcolorbox}[
    colframe=black!50!white,
    colback=blue!5!white,
    boxrule=0.5mm,
    sharp corners,
    boxsep=0.5mm,
    top=0.5mm,
    bottom=0.5mm,
    right=0.25mm,
    left=0.1mm
]
    \begingroup
    \setlength{\abovedisplayskip}{0pt}
    \setlength{\belowdisplayskip}{0pt}
    find $U \in  H^1(0,T;H^1_0(\Omega)) \cap H^2_{0,\bullet}(0,T;L^2(\Omega))$ such that \vspace{0.2cm}
    \begin{equation} \label{eq:55}
        \A(U, \L_T W) = (F, \partial_t \L_T W)_{L^2(Q_T)} \vspace{0.2cm}
    \end{equation}
    for all $W \in H^1(0,T;H_0^1(\Omega)) \cap  H_{0,\bullet}^2(0,T;L^2(\Omega))$.
    \endgroup
\end{tcolorbox}

\noindent
We define the norm
\begin{equation} \label{eq:56}
\begin{aligned} 
    \| U \|^2_{\mathcal{V}(Q_T)} := \| \partial_t U \|_{L^2(Q_T)}^2 + \| c \nabla_{\bx} U \|^2_{L^2(Q_T)}, 
\end{aligned}
\end{equation}
in which we prove, for the form~$\A$, a coercivity property analogous to~\eqref{eq:15}.
\begin{proposition}[Coercivity in the norm~\eqref{eq:56}] \label{prop:42}
We have 
\begin{equation*}
    \A(U,\L_T U) \ge \frac{1}{2eT} \| U \|^2_{\mathcal{V}(Q_T)} \qquad \text{for all}\ \ U \in  H^1(0,T;H_0^1(\Omega)) \cap H^2_{0,\bullet}(0,T;L^2(\Omega)).
\end{equation*}
\end{proposition}
\begin{proof}
We readily compute, using property~\textit{2.} and then \textit{4.}--\textit{7.} of Proposition \ref{prop:24},
\begin{align*}
    \A(U, \L_T U)  =& (\partial_t^2 U, \partial_t \L_T U)_{L^2(Q_T)} + \| \partial_t U(\cdot,0)\|_{L^2(\Omega)}^2 + (c^2 \nabla_{\bx} U, \nabla_{\bx} \partial_t \L_T U)_{L^2(Q_T)}
    \\ = & \frac{1}{2T} (\partial_t U, \partial_t \L_T U)_{L^2(Q_T)} + \frac{1}{2e} \| \partial_t U(\cdot,T)\|_{L^2(\Omega)}^2 \\ 
    & + \frac{1}{2T} \int_0^T e^{-t/T}\|c\nabla_{\bx} U(\cdot, t)\|_{L^2(\Omega)}\, \dd t + \frac{1}{2e} \| c\nabla_{\bx} U(\cdot,T)\|_{L^2(\Omega)}^2
    \\ \ge & \frac{1}{2eT} \left(\|\partial_t U\|_{L^2(Q_T)}^2 + \| c\nabla_{\bx} U\|_{L^2(Q_T)}^2\right),
\end{align*}
which proves the result.
\end{proof}

\noindent
The following well-posedness result, with a stability estimate analogous to~\eqref{eq:12}, is a direct consequence.

\begin{theorem} \label{thm:43}
Let us assume the regularity on the datum~$F$ as in~\eqref{eq:52}. Then, there exists a unique solution of problem~\eqref{eq:55}, and it satisfies
\begin{equation*}
    \| U \|_{\mathcal{V}(Q_T)}  \le 2eT\, \| F\|_{L^2(Q_T)}.
\end{equation*}
\end{theorem}
\begin{proof}
The existence of a solution $U$ of problem~\eqref{eq:55} with the required regularity follows from Proposition~\ref{prop:41} and the fact that the solution therein is also a solution of~\eqref{eq:55}. The stability estimate follows from the coercivity in Proposition~\ref{prop:42}, the Cauchy-Schwarz inequality, and \textit{8.} of Proposition \ref{prop:24}:
\begin{align*}
    \frac{1}{2eT} \| U \|^2_{\mathcal{V}(Q_T)} \le \A(U,\L_T U) &= (F, \partial_t \L_T U)_{L^2(Q_T)} \le \| F\|_{L^2(Q_T)} \| \partial_t \L_T U \|_{L^2(Q_T)}
    \\ & \le \| F\|_{L^2(Q_T)} \| \partial_t U \|_{L^2(Q_T)} \le \| F\|_{L^2(Q_T)} \| U \|_{\mathcal{V}(Q_T)}.
\end{align*}
The stability also implies uniqueness, and the proof is complete.
\end{proof}

\section{Discretization of the full wave problem} \label{sec:5}
The discrete counterpart of~\eqref{eq:55} is readily obtained using appropriate discrete subspaces. For the spatial discretization, let us consider a discrete space~$V^{p_{\bx}}_{h_{\bx}}(\Omega) \subset H_0^1(\Omega)$ depending on a spatial discretization parameter~$h_{\bx}$ and a polynomial degree~$p_{\bx}$, e.g., piecewise polynomial continuous functions or multivariate B-splines \cite{DeBoor2001} of degree~$p_{\bx}$, over a mesh of~$\Omega$ with mesh size~$h_{\bx}$. For the temporal discretization, as in the previous section, we consider a space~$S_{h_t}^{p_t}(0,T) \subset H_{0,\bullet}^2(0,T)
$. Then, we discretize both the trial and test functions in the space--time tensor product space
\[
Q_{\bh}^{\bp}(Q_T) := V^{p_{\bx}}_{h_{\bx}}(\Omega) \otimes S_{h_t}^{p_t}(0,T)\ \ \subset \ H^1(0,T;H_0^1(\Omega)) \cap H^2_{0,\bullet}(0,T;L^2(\Omega)).
\]
 
\noindent
The conforming discretization of \eqref{eq:55} in~$Q_{\bh}^{\bp}(Q_T)$ reads as follows:

\begin{tcolorbox}[
    colframe=black!50!white,
    colback=blue!5!white,
    boxrule=0.5mm,
    sharp corners,
    boxsep=0.5mm,
    top=0.5mm,
    bottom=0.5mm,
    right=0.25mm,
    left=0.1mm
]
    \begingroup
    \setlength{\abovedisplayskip}{0pt}
    \setlength{\belowdisplayskip}{0pt}
    \begin{equation} \label{eq:57}
    \text{find~} U_{\bh} \in Q_{\bh}^{\bp}(Q_T) \text{~such that~} \A(U_{\bh}, \L_T W_{\bh}) = (F, \partial_t \L_T W_{\bh})_{L^2(Q_T)} \text{~for all~} W_{\bh} \in Q_{\bh}^{\bp}(Q_T).
    \end{equation}
    \endgroup
\end{tcolorbox}
\noindent
The explicit expression of the discrete space--time variational formulation in~\eqref{eq:57} is
\begin{equation} \label{eq:58}
\begin{aligned}
    \int_0^T \int_\Omega \partial_t^2 U_{\bh}(\bx,t) \partial_t W_{\bh}(\bx,t) e^{-t/T} \dd \bx \, \dd t + \int_\Omega \partial_t U_{\bh}(\bx,0) \partial_t W_{\bh}(\bx,0) \, \dd \bx
    \\  \hspace{-2cm} + \int_0^T \int_\Omega c^2(\bx)\nabla_{\bx} U_{\bh}(\bx,t) \cdot\nabla_{\bx} \partial_t W_{\bh}(\bx,t) e^{-t/T} \, \dd \bx \, \dd t 
    \\  
    = \int_0^T \int_\Omega F(\bx,t) \partial_t W_{\bh}(\bx,t) e^{-t/T} \, \dd \bx \, \dd t.
\end{aligned}
\end{equation}

\noindent
The well-posedness of the discrete problem~\eqref{eq:57} is an immediate consequence of the coercivity property in Proposition~\ref{prop:42}. 
\begin{proposition} \label{prop:51}
Let the source satisfy $F \in L^2(Q_T)$. Then, there exists a unique solution to problem \eqref{eq:57}, and it satisfies
\begin{equation*}
    \| U_{\bh} \|_{\mathcal{V}(Q_T)} \le 2eT \| F\|_{{L^2(Q_T)}}.
\end{equation*}
\end{proposition}
\begin{proof}
The uniqueness of the solution follows from the coercivity in Proposition~\ref{prop:42}, while existence follows from uniqueness, as~\eqref{eq:57} is a square, finite dimensional linear system. The stability estimate is derived exactly as in Theorem \ref{thm:43}.
\end{proof}

\subsection{Error estimates under Assumption~\ref{ass:38}}
In this section, we derive error estimates under Assumption~\ref{ass:38} on~$S_h^p(0,T)$, namely in the case of quasi-optimal convergence for the ODE problem; see Section~\ref{sec:32}. We introduce two projection operators for space--time functions in Section~\ref{sec:511} and, in Section~\ref{sec:512}, we state the assumptions we make on the discretization spaces, state projection error estimates, and conclude with the error estimates for the space--time method in~\eqref{eq:57}. The approach we adopt is similar to that in \cite{DongMascottoWang2024, Gomez2025, FerrariPerugiaZampa2025}.

\subsubsection{Projection operators}\label{sec:511}
We introduce the following projection operators for space--time functions:
\begin{itemize}
\item 
$\Pg : L^2(0,T;H^1(\Omega)) \to V_{h_{\bx}}^{p_{\bx}}(\Omega) \otimes L^2(0,T)$ defined by
\begin{equation} \label{eq:59}
    (c^2 \nabla_{\bx} (\Pg - \Id) V, \nabla_{\bx} V_{h_{\bx}})_{L^2(Q_T)} = 0 \quad \text{for all~} V_{h_{\bx}} \in V_{h_{\bx}}^{p_{\bx}}(\Omega) \otimes L^2(0,T),
\end{equation}
\item 
$\Ppt : H^2(0,T;L^2(\Omega)) \to L^2(\Omega) \otimes S^{p_t}_{h_t}(0,T)$ defined by 
\begin{equation} \label{eq:60}
    (\partial_t^2 (\Ppt-\Id) W, \partial_t \L_T W_{h_t})_{L^2(Q_T)} + (\partial_t (\Ppt - \Id) W(\cdot,0), \partial_t W_{h_t}(\cdot,0))_{L^2(\Omega)} = 0 
\end{equation}
for all~$W_{h_t} \in L^2(\Omega) \otimes S_{h_t}^{p_t}(0,T)$ .
\end{itemize}
\noindent
With an abuse of notation, for~$\Ppt$ defined in~\eqref{eq:60} for space--time functions, we use the same name as for the projector defined in~\eqref{eq:40} for functions in time only. 

\noindent
The operator~$\Pg$ is well-defined, from the Lax-Milgram lemma. The coercivity estimate
\begin{equation*}
    (\partial_t^2 W_{h_t}, \partial_t \L_T W_{h_t})_{L^2(Q_T)} + \|\partial_t W_{h_t}(\cdot,0)\|^2_{L^2(\Omega)} \ge \frac{1}{2eT} \|\partial_t W_{h_t}\|^2_{L^2(Q_T)}
\end{equation*}
for all $W_{h_t} \in L^2(\Omega) \otimes S_{h_t}^{p_t}(0,T)$, which follows from \emph{4.} and~\emph{2.} of Proposition~\ref{prop:24}, together with the tensor product structure, the finite dimensionality in time, and coercivity and continuity in space, implies that also~$\Ppt$ is well-defined.
Furthermore, $\Pg$ and~$\Ppt$ satisfy 
\[
\Pg\Ppt=\Ppt\Pg, \qquad \partial_t\Pg=\Pg\partial_t, \qquad \nabla_{\bx}\Ppt=\Ppt\nabla_{\bx}.
\]

\noindent
We prove the following auxiliary result.
\begin{lemma} \label{lem:52}
Assume the regularity on the datum~$F$ as in~\eqref{eq:52}. Let~$U$ be the unique solution to~\eqref{eq:55}, and let~$U_{\bh}$ be the unique solution to the discrete problem~\eqref{eq:57}. Suppose also that
\begin{equation} \label{eq:61}
    \nabla_{\bx} \cdot (c^2 \nabla_{\bx} U) \in H^2(0,T;L^2(\Omega)).
\end{equation}
Then, it holds 
\begin{align*}
    \| U_{\bh} - \Pg \Ppt U \|_{\mathcal{V}(Q_T)} & \le 
      2eT \bigl(\| (\Id-\Pg) \partial_t^2 U \|_{{L^2(Q_T)}} + \| (\Id-\Ppt) \nabla_{\bx} \cdot (c^2 \nabla_{\bx} U) \|_{L^2(Q_T)}\bigr).
\end{align*}
\end{lemma}
\begin{proof}
We use the coercivity in Proposition \ref{prop:42} and the Galerkin orthogonality and obtain, for any $W_{\bh} \in Q_{\bh}^{\bp}(Q_T)$,
\begin{equation} \label{eq:62}
\begin{aligned}
 \| U_{\bh} - W_{\bh} \|^2_{\mathcal{V}(Q_T)} 
   & \le 2eT \A(U_{\bh} - W_{\bh}, \L_T (U_{\bh} - W_{\bh}))
   \\ &  = 2eT \A(U-W_{\bh}, \L_T (U_{\bh}- W_{\bh})).
\end{aligned}
\end{equation}
Using the definitions of~$\A$ in \eqref{eq:53} and $\L_T$ in \eqref{eq:54}, we write
\begin{equation} \label{eq:63}
\begin{aligned}
     \A((U-W_{\bh}, \L_T (U_{\bh} - W_{\bh})) 
     = & (\partial_t^2 (U-W_{\bh}), \partial_t \L_T (U_{\bh}-W_{\bh}))_{L^2(Q_T)} \\ & + (\partial_t (U-W_{\bh})(\cdot,0), \partial_t (U_{\bh}-W_{\bh})(\cdot,0))_{L^2(\Omega)}
     \\ & + (c^2\nabla_{\bx} (U-W_{\bh}), \nabla_{\bx} \partial_t \L_T (U_{\bh}-W_{\bh}))_{L^2(Q_T)}
    \\ &  =: I_1 + I_2 + I_3.
\end{aligned}
\end{equation}
Now, we take~$W_{\bh} = \Pi^\nabla_{h_{\bx}} \Ppt U$ and compute, using the commutativity of the projectors and definition of $\Ppt$,
\begin{align*}
    I_1+I_2 &  = (\partial_t^2 (\Id - \Ppt \Pg)U, \partial_t \L_T (U_{\bh}-W_{\bh}))_{L^2(Q_T)} 
    \\ & \quad + (\partial_t (\Id - \Ppt \Pg)U(\cdot,0), \partial_t (U_{\bh}-W_{\bh})(\cdot,0))_{L^2(\Omega)}
    \\ & = (\partial_t^2 (\Id-\Pg)U, \partial_t \L_T (U_{\bh}-W_{\bh}))_{L^2(Q_T)} 
    \\ & \quad + (\partial_t (\Id-\Pg)U(\cdot,0), \partial_t (U_{\bh}-W_{\bh})(\cdot,0))_{L^2(\Omega)}.
\end{align*}
Then, since $\partial_t$ commutes with~$\Pg$, and~$\partial_t U(\cdot,0) = 0$ in $L^2(\Omega)$, we conclude
\begin{align*}
    I_1+I_2 = (\partial_t^2 (\Id-\Pg)U, \partial_t \L_T (U_{\bh}-W_{\bh}))_{L^2(Q_T)}.
\end{align*}
With the Cauchy-Schwarz inequality, the commutativity of $\Pg$ with $\partial_t^2$, property~\emph{2.} of Proposition~\ref{prop:24}, and the definition of the norm in \eqref{eq:56}, we obtain
\begin{equation} \label{eq:64}
\begin{aligned}
    I_1+I_2 \le \|(\Id -\Pg)\partial_t^2 U\|_{L^2(Q_T)} \| U_{\bh} - W_{\bh}\|_{\mathcal{V}(Q_T)}.
\end{aligned}
\end{equation}
\noindent
For~$I_3$, using the definition of $\Pg$, integration by parts in space with the regularity assumption in~\eqref{eq:61}, and the commutativity of $\Ppt$ with $\nabla_{\bx} \cdot (c^2 \nabla_{\bx}\cdot )$, we deduce
\begin{align*}
    I_3 &= (c^2 \nabla_{\bx} (\Id - \Pi^\nabla_{h_{\bx}} \Ppt) U, \nabla_{\bx} \partial_t \L_T (U_{\bh} - W_{\bh}))_{L^2(Q_T)}
    \\ & = (c^2 \nabla_{\bx} (\Id - \Ppt)U, \nabla_{\bx} \partial_t \L_T (U_{\bh} - W_{\bh}))_{L^2(Q_T)}
    \\ & = - (\nabla_{\bx} \cdot(c^2 \nabla_{\bx} (\Id-\Ppt) U),\partial_t \L_T (U_{\bh} - W_{\bh}))_{L^2(Q_T)}
    \\ & = - ((\Id-\Ppt)\nabla_{\bx} \cdot(c^2 \nabla_{\bx} U), \partial_t \L_T (U_{\bh} - W_{\bh}))_{L^2(Q_T)}.
\end{align*}
Using again the Cauchy-Schwarz inequality and property~\emph{2.} of Proposition~\ref{prop:24}, we obtain
\begin{equation} \label{eq:65}
\begin{aligned}
    I_3 \le \|(\Id - \Ppt) \nabla_{\bx} \cdot (c^2 \nabla_{\bx} U) \|_{L^2(Q_T)} \| U_{\bh} - W_{\bh} \|_{\mathcal{V}(Q_T)}. 
     \end{aligned}
\end{equation}
By substituting the estimates~\eqref{eq:64} and \eqref{eq:65} for~$I_1$--$I_3$ into~\eqref{eq:63}, and combining with~\eqref{eq:62}, we obtain the desired claim.
\end{proof}

\subsubsection{Convergence rates} \label{sec:512}
For the finite dimensional space~$S^{p_t}_{h_t}(0,T) \subset H_{0,\bullet}^2(0,T)$, we assume that Assumption~\ref{ass:32} and Assumption~\ref{ass:38} are satisfied, so that the properties established in Section~\ref{sec:32} are valid.
For~$V^{p_{\bx}}_{h_{\bx}}(\Omega)$, we make the following assumption, which is satisfied, for instance, for conforming finite elements on shape-regular, simplicial meshes (see, e.g.~\cite[Theorem 22.6]{ErnGuermond2021a}). From here on, we use the notation~``$\apprle$'' for~``$\le C$'', whenever the constant~$C>0$ is independent of~$h_{\bx}$ and~$h_t$. 

\begin{assumption} \label{ass:53}
The finite dimensional space~$V^{p_{\bx}}_{h_{\bx}}(\Omega) \subset H_0^1(\Omega)$ satisfies the following approximation property for some $p_{\bx} \in \mathbb{N}$, $p_{\bx} \ge 1$: for every~$W$ sufficiently smooth, there is a function~$\widetilde{W}_{h_{\bx}} \in V^{p_{\bx}}_{h_{\bx}}(\Omega)$ satisfying 
\begin{align*}
     \| \nabla_{\bx} W - \nabla_{\bx} \widetilde{W}_{h_{\bx}} \|_{L^2(\Omega)} & \apprle h_{\bx}^{p_{\bx}}. 
\end{align*}
\end{assumption}
\noindent
We have the following projection error estimates.
\begin{lemma} \label{lem:54}
Under Assumptions~\ref{ass:32}, \ref{ass:38}, and~\ref{ass:53}, for any function~$U$ sufficiently smooth, it holds
\begin{align} 
    \| (\Id - \Pg) \partial_t^2 U \|_{{L^2(Q_T)}} & \apprle h_{\bx}^{p_{\bx}+\sigma},
\label{eq:66}
    \\ \| (\Id - \Ppt) \nabla_{\bx} \cdot (c^2 \nabla_{\bx} U) \|_{L^2(Q_T)} & \apprle h_t^{p_t+1},
\label{eq:67}
    \\ \| \partial_t (\Id - \Pi^{\nabla}_{h_{\bx}} \Ppt )U\rVert_{{L^2(Q_T)}} & \apprle h_t^{p_t}  + h^{p_{\bx}+\sigma}_{\bx}, 
\label{eq:68}
    \\ \lVert (\Id - \Pg\Ppt) U \rVert_{{L^2(Q_T)}} & \apprle h_t^{p_t+1} +  h_{\bx}^{p_{\bx}+\sigma},
\label{eq:69}
    \\ \lVert c\nabla_{\bx} (\Id - \Pg\Ppt)U\rVert_{{L^2(Q_T)}} & \apprle h^{p_t+1}_t  + h_{\bx}^{p_{\bx}},
\label{eq:70}
\end{align}
where~$\sigma=0$ in general, and~$\sigma=1$ if the standard elliptic regularity assumption in space is satisfied.
\end{lemma}
\begin{proof}
Assumption~\ref{ass:53} implies that the $\Pg$ satisfies the following approximation property:
\begin{equation} \label{eq:71}
    \| c \nabla_{\bx} (\Id -\Pg) U\|_{L^2(Q_T)} = \inf_{W_{h_{\bx}} \in V_{h_{\bx}}^{p_{\bx}}(\Omega) \otimes L^2(0,T)} \| c\nabla_{\bx} (U-W_{h_{\bx}}) \|_{L^2(Q_T)} \apprle h_{\bx}^{p_{\bx}}.
\end{equation}
Then, if the elliptic regularity in space is satisfied, a standard duality argument yields
\begin{equation} \label{eq:72}
    \| (\Id -\Pg) U\|_{L^2(Q_T)} \apprle h_{\bx} \| c \nabla_{\bx} (\Id -\Pg) U\|_{L^2(Q_T)} \apprle h_{\bx}^{p_{\bx}+1}. 
\end{equation}
However, if elliptic regularity is not satisfied, we use the Poincar\'e's inequality in space and obtain 
\begin{equation} \label{eq:73}
    \| (\Id -\Pg) U\|_{L^2(Q_T)} \apprle \| c \nabla_{\bx} (\Id -\Pg) U\|_{L^2(Q_T)} \apprle h_{\bx}^{p_{\bx}}. 
\end{equation}
%
Inequality \eqref{eq:66} is then a consequence of \eqref{eq:72} and \eqref{eq:73}. \\
The space--time analogue of Corollary \ref{cor:312} gives, for the projector $\Ppt$ defined in \eqref{eq:60},
\begin{equation} \label{eq:74}
    \| \partial_t^\ell U - \partial_t^\ell \Pi_h^{\partial_t^2} U\|_{L^2(Q_T)} \apprle h_t^{p_t+1-\ell},\qquad \ell=0,1,2.
\end{equation}
%
Inequality \eqref{eq:67} follows from \eqref{eq:74} with $\ell=0$. \\
\noindent
In the subsequent steps of the proof, we repeatedly use the following identities, which follow from the commutativity of the projectors:
\begin{align*}
    \Id - \Pg \Ppt = \Id -  \Ppt\Pg = \Pg(\Id - \Ppt) + \Id - \Pg.
\end{align*}
%
\noindent
To prove \eqref{eq:68}, with the commutativity of $\Pg$ and $\partial_t$ and with the Poincar\'e inequality in space, we compute
\begin{align*}
    \lVert \partial_t(\Id - \Pg \Ppt)U\rVert_{{L^2(Q_T)}} & \leq \lVert \partial_t (\Id - \Pg)U \rVert_{{L^2(Q_T)}} + \lVert\partial_t \Pg (\Id - \Ppt)U \rVert_{{L^2(Q_T)}} \\
    & \apprle \lVert (\Id - \Pg) \partial_t U \rVert_{{L^2(Q_T)}} + \lVert c \nabla_{\bx} \Pg \partial_t (\Id - \Ppt) U \rVert_{{L^2(Q_T)}}
    \\ & \leq \lVert (\Id - \Pg) \partial_t U \rVert_{{L^2(Q_T)}} + \lVert \partial_t (\Id - \Ppt) c \nabla_{\bx} U \rVert_{{L^2(Q_T)}},
\end{align*}
where, in the last step, we used the stability of the elliptic projector $\Pg$ in $H^1_0(\Omega)$ and the commutativity of~$\nabla_{\bx}$ with~$\partial_t$ and~$\Ppt$. Then, \eqref{eq:68} follows from the approximation properties~\eqref{eq:72}, \eqref{eq:73}, and~\eqref{eq:74} with~$\ell=1$.

\noindent
For inequality \eqref{eq:69}, the Poincar\'e inequality in space, the stability of $\Pg$ in $H_0^1(\Omega)$, and the commutativity properties imply
\begin{align*}
    \lVert (\Id - \Pg\Ppt) U \rVert_{{L^2(Q_T)}} & \leq \lVert \Pg(\Id - \Ppt) U \rVert_{{L^2(Q_T)}} + \lVert (\Id - \Pg) U\rVert_{{L^2(Q_T)}} 
    \\ & \apprle \lVert c \nabla_{\bx} \Pg(\Id - \Ppt) U\rVert_{{L^2(Q_T)}} + \lVert (\Id-\Pg) U\rVert_{{L^2(Q_T)}} 
    \\ & \apprle \lVert (\Id - \Ppt) c \nabla_{\bx} U\rVert_{{L^2(Q_T)}} + \lVert (\Id-\Pg) U\rVert_{{L^2(Q_T)}},
\end{align*}
and~\eqref{eq:69} follows from the approximation properties properties~\eqref{eq:72}, \eqref{eq:73}, and~\eqref{eq:74} with~$\ell=0$.

\noindent
Finally, for~\eqref{eq:70}, we proceed similarly and compute, with the stability of $\Pg$ in $H_0^1(\Omega)$ and the commutativity properties,
\begin{align*}
     \lVert c \nabla_{\bx} (\Id - \Pg \Ppt)U\rVert_{{L^2(Q_T)}} 
    &  \leq \lVert c \nabla_{\bx} \Pg (\Id - \Ppt)U \rVert_{{L^2(Q_T)}} + \lVert c \nabla_{\bx} (\Id - \Pg)U \rVert_{{L^2(Q_T)}}
    \\ &  \leq \lVert (\Id - \Ppt)  c\nabla_{\bx} U \rVert_{{L^2(Q_T)}} + \lVert c \nabla_{\bx} (\Id - \Pg)U \rVert_{{L^2(Q_T)}},
\end{align*}
and conclude with \eqref{eq:71} and \eqref{eq:74} with~$\ell=0$.
\end{proof}
\noindent
We now establish the error estimates for the space--time method in~\eqref{eq:57}.
\begin{theorem} \label{th:55}
Assume the regularity on the datum~$F$ as in~\eqref{eq:52}. Let~$U$ be the unique solution to~\eqref{eq:55}, and let~$U_{\bh}$ be the unique solution to the discrete problem~\eqref{eq:57}. Under Assumptions~\ref{ass:32}, \ref{ass:38}, and~\ref{ass:53}, for $U$ sufficiently smooth, it holds

\begin{align*}
    \| \partial_t U- \partial_t U_{\bh} \|_{L^2(Q_T)}  & \apprle h_t^{p_t} + h_{\bx}^{p_{\bx}+\sigma},
    \\\| \nabla_{\bx} U- \nabla_{\bx} U_{\bh} 
     \|_{L^2(Q_T)} & \apprle h_t^{p_t+1}+ h_{\bx}^{p_{\bx}},
    \\ \| U- U_{\bh} \|_{L^2(Q_T)} & \apprle h_t^{p_t+1} + h_{\bx}^{p_{\bx}+\sigma}.
\end{align*}
where~$\sigma=0$ in general, and~$\sigma=1$ if the standard elliptic regularity assumption in space is satisfied.
\end{theorem}
\begin{proof}
The three inequalities can be derived using the triangle inequality, the definition of the norm in~\eqref{eq:56}, Lemma \ref{lem:52}, and Lemma \ref{lem:54}, with the Poincar\'e inequality used specifically for the third one.
\end{proof}

\subsection{Numerical validation}
As a test case, we consider the one–dimensional spatial domain~$\Omega = (0,1)$, the final time~$T=1$, and the wave velocity~$c\equiv 1$. The data are chosen such that the exact solution of~\eqref{eq:58} is
\begin{equation} \label{eq:75}
    U(x,t) = \sin(x \pi) \sin^2 \left(\frac{5}{4} \pi t \right), \quad (x,t) \in (0,1) \times (0,1).
\end{equation}
For the space--time Galerkin discretization, we choose the discrete spaces generated by B-splines over uniform meshes for both~$V_{h_{\bx}}^{p_{\bx}}(0,1)$ and ~$S_{h_t}^{p_t}(0,1)$, using the same polynomial degree~$p = p_{\bx} = p_t$, but allowing for different mesh sizes~$h_{\bx}$ and~$h_t$. For $V_{h_{\bx}}^{p_{\bx}}(0,1)$, we always employ \emph{maximal regularity splines}, while, for $V_{h_t}^{p_t}(0,1)$, we vary the regularity in the numerical tests.

\noindent
In the first test, we demonstrate the unconditional stability of the proposed method by calculating the relative errors in the~$H^2$ seminorm,~$H^1$ norm, and~$L^2$ norm, while fixing~$h_t = 0.125$ and progressively decreasing~$h_{\bx}$. These norms and seminorms are defined as follows:
\begin{equation} \label{eq:76}
\begin{aligned}
    \| U \|_{L^2(Q_T)} := & \left(\int_0^T \int_\Omega U^2(\bx,t) \, \dd \bx \, \dd t\right)^{\frac{1}{2}} \hspace{-0.15cm}, \, \, \, \| U \|_{H^1(Q_T)} := \| \partial_t U \|_{L^2(Q_T)} + \| c\nabla_{\bx} U \|_{L^2(Q_T)} 
    \\ & \hspace{-0.2cm} | U |_{H^2(Q_T)} := \| \partial_t^2 U \|_{L^2(Q_T)} + \|\text{div}(c^2 \nabla_{\bx} U) \|_{L^2(Q_T)} + \| c\nabla_{\bx} \partial_t U \|_{L^2(Q_T)}.
    \end{aligned}
\end{equation}
In Figure~\ref{fig:4}, we present the relative errors in these norms/seminorms for~$p = 2, 3, 4$ and \emph{maximal regularity} splines in both space and time. No instabilities are observed in the results.

\begin{figure}[ht]
\centering
\begin{minipage}{0.325\textwidth}
    \centering
   \begin{tikzpicture}
    \begin{loglogaxis}[
        legend style={draw=none},
        title =~$H^2$-error,
        xlabel =~$h_t/h_{\bx}$,
        width=0.9\textwidth,
        height=0.75\textwidth]
        \pgfplotstableread{
        h  error
        1     0.247481413739432
        2     0.238441840732300
        4     0.233508350309428
        8     0.232045544723915
        16    0.231734762504643
        32    0.231723474160233
        64    0.231718175142824
        128   0.231717504533943
        256   0.231717374937925
        512   0.231717295848743
        } \datatable
        \addplot[mark=*, blue] table[x=h, y=error] \datatable;
        \pgfplotstableread{
        h  error
        1     0.027998171957310
        2     0.027339759917056
        4     0.027243865461405
        8     0.027237839628720
        16    0.027237115298263
        32    0.027237028673317
        64    0.027237019568328
        128   0.027237017355619
        256   0.027237017063335
        512   0.027237017014680
        } \datatable
        \addplot[mark=square*, red] table[x=h, y=error] \datatable;
        \pgfplotstableread{
        h  error
        1     0.004426327674723
        2     0.004397042395973
        4     0.004396126340963
        8     0.004396114358015
        16    0.004396113467925
        32    0.004396113403149
        64    0.004396113400058
        128   0.004396113401344
        256   0.004396113402024
        512   0.004396113402522
        } \datatable
        \addplot[mark=diamond*, darkgreen, mark size=3] table[x=h, y=error] \datatable;
    \end{loglogaxis}
    \end{tikzpicture}
\end{minipage}
\begin{minipage}{0.325\textwidth}
    \centering
    \begin{tikzpicture}
    \begin{loglogaxis}[
        legend style={draw=none},
        width=0.9\textwidth,
        title =~$H^1$-error,
        xlabel =~$h_t/h_{\bx}$,
        height=0.75\textwidth]
        \pgfplotstableread{
        h  error
        1     0.047478679385498
        2     0.046977801834104
        4     0.046950331310437
        8     0.046947589949453
        16    0.046947412084121
        32    0.046947389678381
        64    0.046947388387963
        128   0.046947388423246
        256   0.046947388419273
        512   0.046947388457831
        } \datatable
        \addplot[mark=*, blue] table[x=h, y=error] \datatable;
        \pgfplotstableread{
        h  error
        1     0.004436562138127
        2     0.004409019376268
        4     0.004406253469785
        8     0.004406188436507
        16    0.004406187385570
        32    0.004406187297242
        64    0.004406187305489
        128   0.004406187307586
        256   0.004406187308413
        512   0.004406187311675
        } \datatable
        \addplot[mark=square*, red] table[x=h, y=error] \datatable;
        \pgfplotstableread{
        h  error
        1     0.000759890698107
        2     0.000758352183741
        4     0.000758348993527
        8     0.000758348741053
        16    0.000758348746429
        32    0.000758348747578
        64    0.000758348747576
        128   0.000758348747649
        256   0.000758348747567
        512   0.000758348747922
        } \datatable
        \addplot[mark=diamond*, darkgreen, mark size=3] table[x=h, y=error] \datatable;
    \end{loglogaxis}
    \end{tikzpicture}
\end{minipage}
\begin{minipage}{0.325\textwidth}
    \centering
     \begin{tikzpicture}
    \begin{loglogaxis}[
        legend style={draw=none},
        width=0.9\textwidth,
        title =~$L^2$-error,
        xlabel =~$h_t/h_{\bx}$,
        height=0.75\textwidth]
        \pgfplotstableread{
        h  error
        1     0.012285195418693
        2     0.012285152618984
        4     0.012285138066227
        8     0.012285137739979
        16    0.012285137721838
        32    0.012285137717530
        64    0.012285137714858
        128   0.012285137715135
        256   0.012285137711463
        512   0.012285137708621
        } \datatable
        \addplot[mark=*, blue] table[x=h, y=error] \datatable;
        \pgfplotstableread{
        h  error
        1     0.000802468054518438
        2     0.000802301732395376
        4     0.000802300051917742
        8     0.000802300034748272
        16    0.000802300034720438
        32    0.000802300034724396
        64    0.000802300034715466
        128   0.000802300034715966
        256   0.000802300034718397
        512   0.000802300034741443
        } \datatable
        \addplot[mark=square*, red, mark size=2] table[x=h, y=error] \datatable;
        \pgfplotstableread{
        h  error
        1     0.000132740439922740
        2     0.000132713716781978
        4     0.000132713595121259
        8     0.000132713586245648
        16    0.000132713586318472
        32    0.000132713586319023
        64    0.000132713586273442
        128   0.000132713586376798
        256   0.000132713586477945
        512   0.000132713586594247
        } \datatable
        \addplot[mark=diamond*, darkgreen, mark size=3] table[x=h, y=error] \datatable;
    \end{loglogaxis}
    \end{tikzpicture}
\end{minipage}
\caption{Relative errors in the norms/seminorms defined in~\eqref{eq:76}, with maximal regularity splines in both space and time for~$p=2$ (\textcolor{blue}{$\bullet$} marker),~$p=3$ ({\footnotesize{\textcolor{red}{$\blacksquare$}}} marker) and~$p=4$ ({\footnotesize{\textcolor{darkgreen}{\(\blacklozenge\)}}} marker), and the exact solution as in~\eqref{eq:75}. Here,~$h_t = 0.125$ and~$h_{\bx}$ decreases.}
\label{fig:4}
\end{figure}
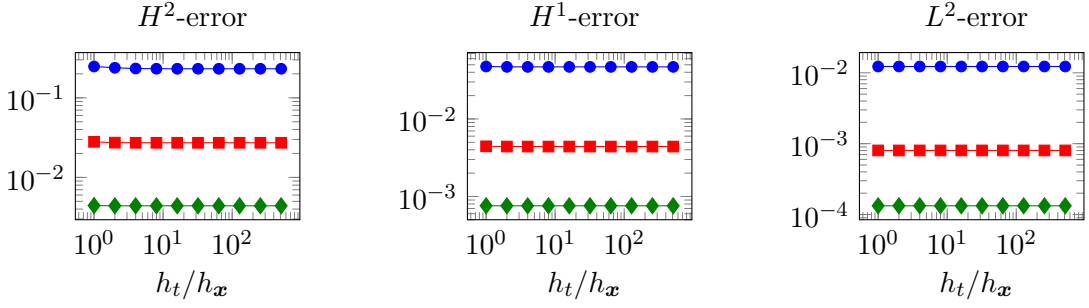

\noindent
In the second test, we investigate the order of convergence in the three norms/seminorms in~\eqref{eq:76}. To this end, we consider two situations: $C^1$-regular B-splines and B-splines with maximal regularity for the discrete spaces in time. In all the tests  we consider~$h_{\bx} = h_t = h$.

\noindent
In Figure~\ref{fig:5}, we report the relative errors for~$p=3$ and $p=4$ and $C^1$-splines in time. As proven in Sections \ref{sec:32} and \ref{sec:33}, quasi-optimal order of convergence is achieved when $p$ is even, and suboptimal by one order when $p$ is odd.

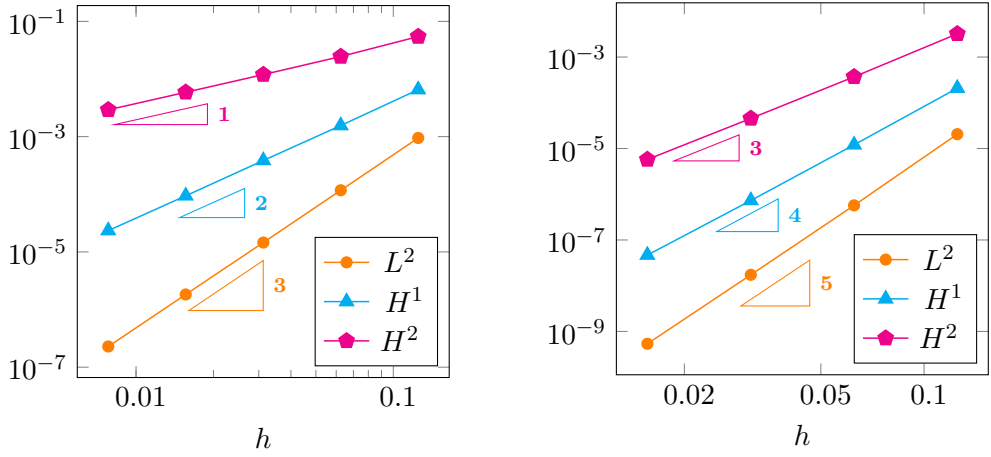
\begin{figure}[h!]
\centering
\begin{minipage}{0.45\textwidth}
\begin{tikzpicture}
\begin{groupplot}[group style={group size=2 by 1},height=6.5cm,width=6.5cm, every axis label={font=\normalsize}]
        
    \nextgroupplot[xmode=log, 
                log x ticks with fixed point,
                xtick={0.01,0.1},
                ytick={0.1,0.001,0.00001,0.0000001,0.00000001},
                xlabel={$h$},
                ymode=log,
                legend pos=south east, 
                legend style={nodes={scale=1, transform shape}},]
        \pgfplotstableread{
            x       y
            0.125     0.000948253971128572
            0.0625    0.000117338697542266
            0.03125   0.000014561655308654
            0.0156    0.000001826994452005
            0.0078    0.000000228412969056
        } \Lduepdue
    \addplot[orange, mark=*, mark options={scale=1}, line width=0.02cm] table[x=x, y=y] \Lduepdue;
        \pgfplotstableread{
            x       y
           0.125    0.006546342449583
           0.0625    0.001557986403072
            0.03125   0.000385684285002
            0.0156    0.000094061652807
            0.0078    0.000023424214679
        } \Hunopdue     
        \addplot[cyan, mark=triangle*, mark options={scale=1.5}, line width=0.02cm] table[x=x, y=y] \Hunopdue;
        \pgfplotstableread{
            x       y
            0.125    0.054522407357225
            0.0625    0.024570654938842
            0.03125   0.011930306834046
            0.0156    0.005884679723263
            0.0078    0.002920609359845
        } \Hduepdue
        \addplot[magenta, mark=pentagon*, mark options={scale=1.5}, line width=0.02cm] table[x=x, y=y] \Hduepdue;

        \logLogSlopeTriangle{0.35}{0.25}{0.68}{1}{magenta};        
        \logLogSlopeTriangle{0.45}{0.175}{0.43}{2}{cyan};        
        \logLogSlopeTriangle{0.5}{0.2}{0.18}{3}{orange};
        
        \legend
        {$L^2$,~$H^1$,$H^2$}
    \end{groupplot}
\end{tikzpicture}
\end{minipage}
\begin{minipage}{0.45\textwidth}
\begin{tikzpicture}
\begin{groupplot}[group style={group size=2 by 1},height=6.5cm,width=6.5cm, every axis label={font=\normalsize}]
        
    \nextgroupplot[xmode=log, 
                log x ticks with fixed point,
                xtick={0.02,0.05,0.1},
                ytick={0.001,0.00001,0.0000001,0.000000001},
                xlabel={$h$},
                ymode=log,
                legend pos=south east, 
                legend style={nodes={scale=1, transform shape}},]
        \pgfplotstableread{
            x       y
            0.125     0.0000205674302721903
            0.0625    0.0000005705451280611
            0.03125   0.0000000173027472390
            0.0156    0.0000000005371203825
        } \Lduepdue
    \addplot[orange, mark=*, mark options={scale=1}, line width=0.02cm] table[x=x, y=y] \Lduepdue;
        \pgfplotstableread{
            x       y
           0.125    0.000208254764253067
           0.0625    0.000012036324943743
        0.03125   0.000000729083631469
            0.0156    0.000000046704299853
        } \Hunopdue     
        \addplot[cyan, mark=triangle*, mark options={scale=1.5}, line width=0.02cm] table[x=x, y=y] \Hunopdue;
        \pgfplotstableread{
            x       y
            0.125    0.003220767870442
            0.0625    0.000369498405268
            0.03125   0.000045456105466
            0.0156    0.000005755010478
        } \Hduepdue
        \addplot[magenta, mark=pentagon*, mark options={scale=1.5}, line width=0.02cm] table[x=x, y=y] \Hduepdue;

        \logLogSlopeTriangle{0.33}{0.175}{0.575}{3}{magenta};       
        \logLogSlopeTriangle{0.435}{0.165}{0.385}{4}{cyan};        
        \logLogSlopeTriangle{0.52}{0.185}{0.185}{5}{orange};
        
        \legend
        {$L^2$,~$H^1$,$H^2$}
    \end{groupplot}
\end{tikzpicture}
\end{minipage}
\caption{Relative errors in the norms/seminorms defined in~\eqref{eq:76} with $C^1$-regular splines in time and maximal regularity splines in space for~$p=3$ (left plot) and~$p=4$ (right plot), by varying the mesh size~$h_{\bx}=h_t=h$, for the exact solution as in~\eqref{eq:75}.}
\label{fig:5}
\end{figure}

\noindent
In Figure~\ref{fig:6}, we report the results for maximal regularity splines in time, for~$p = 2$ and~$p = 3$. The results shown validate 
the quasi-optimal order of convergence conjectured at the beginning of Section \ref{sec:32}.

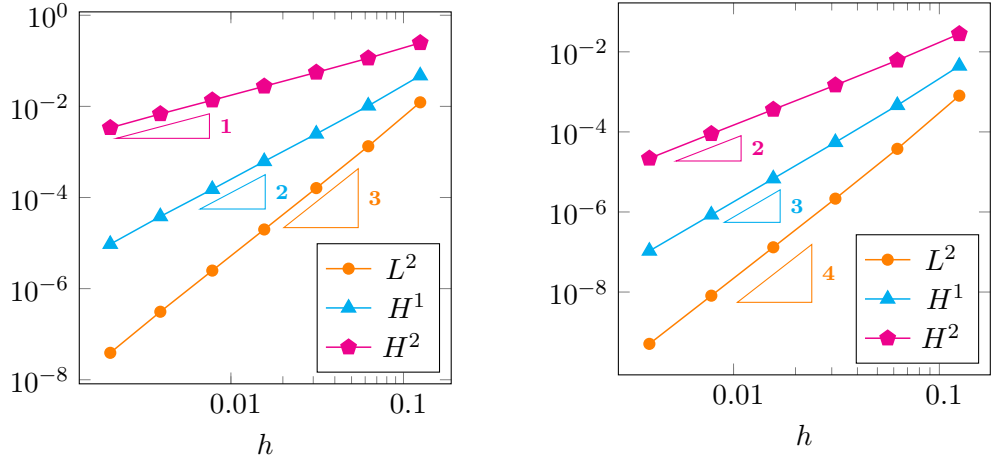
\begin{figure}[h!]
\centering
\begin{minipage}{0.45\textwidth}
\begin{tikzpicture}
\begin{groupplot}[group style={group size=2 by 1},height=6.5cm,width=6.5cm, every axis label={font=\normalsize}]
        
    \nextgroupplot[xmode=log, 
                log x ticks with fixed point,
                xtick={0.01,0.1},
                ytick={1,0.01,0.0001,0.000001,0.00000001},
                xlabel={$h$},
                ymode=log,
                legend pos=south east, 
                legend style={nodes={scale=1, transform shape}},]
        \pgfplotstableread{
            x       y
            0.125     0.012276257552353
            0.0625    0.001333830327579
            0.03125   0.000160870808909
            0.0156    0.000019867024802
            0.0078    0.000002486968232
            0.0039    0.000000310913053
            0.0020    0.000000038985725
        } \Lduepdue
    \addplot[orange, mark=*, mark options={scale=1}, line width=0.02cm] table[x=x, y=y] \Lduepdue;
        \pgfplotstableread{
            x       y
            0.125    0.047210237165267
            0.0625    0.010273676899700
            0.03125   0.002488834541324
            0.0156    0.000620388669461
            0.0078    0.000151658728011
            0.0039    0.000038285131298
            0.0020    0.000009449208366
        } \Hunopdue     
        \addplot[cyan, mark=triangle*, mark options={scale=1.5}, line width=0.02cm] table[x=x, y=y] \Hunopdue;
        \pgfplotstableread{
            x       y
            0.125     0.247606045103127
            0.0625    0.114078981147447
            0.03125   0.055627094070887
            0.0156    0.027545907615786
            0.0078    0.013656442406419
            0.0039    0.006823743369201
            0.0020    0.003388274287899
        } \Hduepdue
        \addplot[magenta, mark=pentagon*, mark options={scale=1.5}, line width=0.02cm] table[x=x, y=y] \Hduepdue;

        \logLogSlopeTriangle{0.35}{0.25}{0.66}{1}{magenta};        
        \logLogSlopeTriangle{0.5}{0.175}{0.47}{2}{cyan};        
        \logLogSlopeTriangle{0.75}{0.2}{0.42}{3}{orange};
        
        \legend
        {$L^2$,~$H^1$,$H^2$}
    \end{groupplot}
\end{tikzpicture}
\end{minipage}
\begin{minipage}{0.45\textwidth}
\begin{tikzpicture}
\begin{groupplot}[group style={group size=2 by 1},height=6.5cm,width=6.5cm, every axis label={font=\normalsize}]
        
    \nextgroupplot[xmode=log, 
                log x ticks with fixed point,
                xtick={0.01,0.1},
                ytick={1,0.01,0.0001,0.000001,0.00000001},
                xlabel={$h$},
                ymode=log,
                legend pos=south east, 
                legend style={nodes={scale=1, transform shape}},]
        \pgfplotstableread{
            x       y
            0.125     0.000802930964193882
            0.0625    0.000037775574775117
            0.03125   0.000002151292859217
            0.0156    0.000000131191524704
            0.0078    0.000000008112117510
            0.0039    0.000000000507729076
        } \Lduepdue
    \addplot[orange, mark=*, mark options={scale=1}, line width=0.02cm] table[x=x, y=y] \Lduepdue;
        \pgfplotstableread{
            x       y
                      0.125    0.004438520134026
           0.0625    0.000460930570275
        0.03125   0.000054917391259
            0.0156    0.000006800540289
            0.0078    0.000000849435412
            0.0039    0.000000106017110 
        } \Hunopdue     
        \addplot[cyan, mark=triangle*, mark options={scale=1.5}, line width=0.02cm] table[x=x, y=y] \Hunopdue;
        \pgfplotstableread{
            x       y
            0.125    0.028026728732756
            0.0625    0.006133156547305
            0.03125   0.001469600328859
            0.0156    0.000362189652102
            0.0078    0.000089556850090
            0.0039    0.000021799293646
        } \Hduepdue
        \addplot[magenta, mark=pentagon*, mark options={scale=1.5}, line width=0.02cm] table[x=x, y=y] \Hduepdue;

        \logLogSlopeTriangle{0.33}{0.175}{0.575}{2}{magenta};       
        \logLogSlopeTriangle{0.435}{0.15}{0.41}{3}{cyan};        
        \logLogSlopeTriangle{0.52}{0.2}{0.195}{4}{orange};
        
        \legend
        {$L^2$,~$H^1$,$H^2$}
    \end{groupplot}
\end{tikzpicture}
\end{minipage}
\caption{Relative errors in the norms/seminorms defined in~\eqref{eq:76} with maximal regularity splines both in space and in time for~$p=2$ (left plot) and~$p=3$ (right plot), by varying the mesh size~$h_{\bx}=h_t=h$, for the exact solution as in~\eqref{eq:75}.}
\label{fig:6}
\end{figure}

\section{Conclusion}
In this paper, we studied a conforming space--time method for the wave equation based on a new second-order variational formulation, with smooth spline discretization in time. The method is proven to be unconditionally stable. 

\noindent
The scheme is obtained by discretizing first an associated ODE in time, with stability relying on two key components: the modification of the test functions through an appropriate isomorphism, and the inclusion of a term that contains the first derivative at the initial time. The resulting discretization of the ODE is stated in~\eqref{eq:21} or, equivalently, in~\eqref{eq:22}. The proposed method yields error estimates with respect to the mesh size that are suboptimal by one order in standard Sobolev norms. However, for certain choices of approximation spaces, it achieves quasi-optimal convergence. In particular, we prove this for $C^1$-regular splines of even polynomial degree, and present numerical evidence suggesting that this may also apply to splines with maximal regularity, irrespective of the degree. Based on numerical findings, our conjecture is that, for splines, quasi-optimal convergence occurs whenever the difference between degree and regularity is odd. A notable case where the conjecture is valid but not proven is that of splines with maximal regularity. However, even if the order of convergence were suboptimal by one and switching from~$p$- to~$(p+1)$-degree maximal regularity splines were required to achieve a convergence rate of~$p$ for the~$H^1$ error, the increase in degrees of freedom would be minimal. Numerical experiments confirm the sharpness of the results.

\noindent
Building on the analysis for the ODE problem, we then studied the numerical scheme for the wave equation stated in~\eqref{eq:57} or, equivalently, in~\eqref{eq:58}. This scheme is proven to be unconditionally stable, for any choice of conforming discrete space. The convergence rates are always quasi-optimal in space, and quasi-optimality in time is achieved under the same conditions as in the ODE case.

\noindent
We remark that, if we remove the exponential weight in~\eqref{eq:22}, the same properties of stability and convergence can be still observed numerically. In that case, for splines with maximal regularity, a theoretical analysis could be performed exploiting the quasi-Toeplitz structure of the system matrix as in~\cite{FerrariFraschini2024,FerrariFraschiniLoliPerugia2024}; however, a more general variational argument following the approach of this paper remains to be developed. Finally, important aspects not yet addressed are the efficient implementation of the space--time method and the possibility to reformulate it as a time-marching scheme as, e.g., in~\cite{Tani2017,LangerZank2021, LoliSangalli2025}, also in view of an assessment of the computational cost and comparative studies.

\section{Acknowledgments}
The authors sincerely thank Lorenzo Mascotto for bringing reference~\cite{Walkington2014} to their attention and for the insightful discussions, from which this research originated.

\noindent
This research was supported by the Austrian Science Fund (FWF) project \href{https://doi.org/10.55776/F65}{10.55776/F65} (IP) and \href{https://doi.org/10.55776/P33477}{10.55776/P33477} (MF, IP). MF is member of the Gruppo Nazionale Calcolo Scientifico-Istituto Nazionale di Alta Matematica (GNCS-INdAM).

\bibliography{mybibliography}{}
\bibliographystyle{plain}

\appendix

\section{Orthogonal polynomials with exponential-like weight}\label{app:A}

In this appendix, we prove some results on polynomials that are orthogonal with respect to exponential-like weight function. In particular, we prove that the moments associated with these polynomials exhibit a behavior similar to those of the Legendre polynomials. These results are employed in deriving error estimates for the non-local projection operator defined in Section~\ref{sec:33}.

\noindent
Consider an interval~$(0,T)$, $T>0$ and, a uniform mesh $\{t_j = jh \mid j = 0,\ldots,N\}$ with mesh size $h = T/N$, and the Legendre polynomials of degree~$r$, $\{L_r^i\}_{r \ge 0}$, over the interval $[t_{i-1},t_i]$, $i=1,\ldots,N$. We fix the normalization condition $L_r^i(t_i)=1$ for all $r \ge 0$. The polynomials~$\{L_r^i\}_{r \ge 0}$ can be obtained as a translation and scaling of the Legendre polynomials $\{L_r\}_{r \ge 0}$ defined in $[0,1]$, which satisfy $L_r(1)=1$, as follows:
\begin{equation} \label{eq:77}
    L^i_r(t) = L_r\left(\frac{t-t_{i-1}}{h}\right) \quad \text{for~} t \in [t_{i-1},t_i]. 
\end{equation}
From~\eqref{eq:77}, we obtain $L_r^i(t_{i-1})=L_r(0) = (-1)^r$, as well as the identity
\begin{equation} \label{eq:78}
    (L_r^i,L_r^i)_{L^2(t_{i-1},t_i)} = \mathcal{O}(h),
\end{equation}
where the implicit constant depends only on~$r$. For a fixed index~$i$, defining $q_k(t) := (t-t_{i-1})^k$, we further obtain
\begin{equation} \label{eq:79}
    (L_r^i,q_k)_{L^2(t_{i-1},t_i)} = h^{k+1} \int_0^1 L_r(t) t^k \, \dd t = \begin{cases}
    0 & k < r,
    \\ \mathcal{O}(h^{k+1}) & k\ge r,
    \end{cases}
\end{equation}
with the implicit constant not depending on the index $i$.

\noindent
Consider now polynomials $\{P_r^i\}_{r \ge 0}$ over the interval $[t_{i-1},t_i]$, which are orthogonal with respect to the scalar product
$(u,v)_w := \int_0^T u(t) v(t) w(t) \dd t$, where~$w(t)$ is a weight function (non-negative, integrable, and with a finite number of zeros). Suppose that the weight function $w(t): [-T,T]\to \R^+$ satisfies the following two assumptions:
\begin{equation} \label{eq:80}
    \begin{aligned}
         \text{for all~} t, s \in [-T,T]\ \text{with}\ t+s\in [-T,T],  &\qquad w(t+s) = w(t)w(s),
        \\  w(t) = 1 + \mathcal{O}(t),  &\qquad \text{as~} |t| \to 0.
    \end{aligned}
\end{equation}
\begin{remark}
The weight $w(t) = e^{-t/T}$ considered in the proof of Proposition \ref{prop:315} satisfies these two assumptions.
\end{remark}
\noindent
We fix the normalization condition $P_r^i(t_i)=1$ for all $r \ge 0$. These polynomials are invariant under translations. Indeed, for $k \ne r$, we compute 
\begin{equation*}
    0 = \int_0^h P^1_r(t) P^1_k(t) w(t) \, \dd t = w(-t_{i-1}) \int_{t_{i-1}}^{t_i} P^1_r(t-t_{i-1}) P_k^1(t-t_{i-1}) w(t) \, \dd t.
\end{equation*}
From this, we deduce 
\begin{equation} \label{eq:81}
    P_r^i(t) = P_r^1(t-t_{i-1}).
\end{equation}
We prove the following lemma.
\begin{lemma} \label{lem:A2}
Let $\{P_r^1\}_{r \ge 0}$ be the orthogonal polynomials in $[0,h]$ with respect to a weight function $w(t)$ satisfying assumptions \eqref{eq:80}, and normalized such that $P_r^1(h)=1$. Then, for all $t \in [0,h]$, 
\begin{equation*}
    P_r^1(t) = L_r(t/h) + \mathcal{O}(h) \quad \text{as~} h \to 0,
\end{equation*}
where $\{L_r\}_{r \ge 0}$ are the Legendre polynomials in $[0,1]$, normalized such that $L_r(1)=1$. 
\end{lemma}
\begin{proof}
The proof proceeds by induction on $r$. For the base case $r=0$, we have $P_0^1(t) \equiv 1$ which matches $L_0(t/h)$. For $r=1$, we make the ansatz $P_1^1(t) = a(t-h)+1$, and compute
\begin{align*}
    0 & = \int_0^h P^1_r(t) P^0_r(t) w(t) \, \dd t = \int_0^h (a(t-h)+1)w(t) \, \dd t = \int_0^h (a(t-h)+1)(1+\mathcal{O}(t)) \, \dd t
    \\ & = h \left(-a\frac{h}{2} +1\right) +  \left(-a\frac{h}{2} +1\right) \mathcal{O}(h^2).
\end{align*}
Solving for $a$, we obtain $a = \frac{2}{h} + \mathcal{O}(1)$, yielding 
\begin{equation*}
    P_1^1(t) = \frac{2}{h}(t-h) + 1 +\mathcal{O}(h).
\end{equation*}
Since the Legendre polynomial~$L_1(t)=2t-1$ satisfies $L_1(t/h) = 2\frac{t}{h} - 1$, we conclude that 
\begin{equation*}
    P_1^1(t) = L_1(t/h) + \mathcal{O}(h).
\end{equation*}   
Now, we assume that the claim holds for all $q \le r$, and we aim to prove it for $q=r+1$. In the interval $[0,h]$, the polynomials $\{P_r^1\}_{r \ge 0}$ are expressed with the three-term recursion 
\begin{equation*}
    P^1_{r+1}(t) = \left(A^h_r t + B^h_r\right) P_r^1(t) - C^h_r P_{r-1}^1(t),
\end{equation*}
where the coefficients $A^h_r, B^h_r$, and $C^h_r$ depends on $r$ and $h$. These coefficients are computed as follows. We impose the condition {$1 = P^1_{r+1}(h)$ and obtain
\begin{equation*}
    A^h_r h + B^h_r - C^h_r=1.
\end{equation*}
}
Then, multiplying by $P_r^1(t)$ and integrating with weight~$w$ over the interval $[0,h]$, we obtain
\begin{equation*}
    B^h_r = -A^h_r \, \frac{\int_0^h t [P_r^1(t)]^2 w(t) \, \dd t}{\int_0^h [P_r^1(t)]^2 w(t) \, \dd t} =: -A_r^h I_r^h.
\end{equation*}
Similarly, multiplying by $P_{r-1}^1(t)$ and integrating over $[0,h]$, we obtain
\begin{equation*}
    C_r^h = A_r^h \, \frac{\int_0^h t P_r^1(t) P^1_{r-1}(t) w(t) \, \, \dd t}{\int_0^h [P^1_{r-1}(t)]^2 w(t) \, \dd t} =: A_r^h J_r^h.
\end{equation*}
Combining these results, we find that 
\begin{equation*}
    A_r^h = \frac{1}{h-I_r^h-J_r^h}.
\end{equation*}
An analogous formula holds for the polynomials $\{L_r(t/h)\}_{r \ge 0}$ with coefficients $\widetilde{A}^h_r, \widetilde{B}^h_r$, and $\widetilde{C}^h_r$, and with the ratios $\widetilde{I}_r^h$ and $\widetilde{J}_r^h$. By the induction hypothesis, using~\eqref{eq:80}, \eqref{eq:79} and~\eqref{eq:78}, we have
\begin{align*}
    \int_0^h t [P_r^1(t)]^2 w(t) \, \dd t  = \int_0^h t [L_r(t/h) + \mathcal{O}(h)]^2 (1+\mathcal{O}(t)) \, \dd t = \int_0^h t [L_r(t/h)]^2 \, \dd t + \mathcal{O}(h^3).
\end{align*}
Similarly, we compute
\begin{align}
    \nonumber \int_0^h t P_r^1(t) P^1_{r-1}(t) w(t) \, \dd t & = \int_0^h t L_r(t/h) L_{r-1}(t/h) \,\dd t + \mathcal{O}(h^3), 
    \\ \nonumber \int_0^h [P_{r-1}^1(t)]^2 w(t) \, \dd t & = \int_0^h [L_{r-1}(t/h)]^2 \, \dd t + \mathcal{O}(h^2),
    \\ \label{eq:82} \int_0^h [P_r^1(t)]^2 w(t) \, \dd t & = \int_0^h [L_r(t/h)]^2 \, \dd t + \mathcal{O}(h^2).
\end{align}
From these relations, we obtain 
\begin{equation*}
    I_r^h = \widetilde{I}_r^h + \mathcal{O}(h^2), \quad J_r^h = \widetilde{J}_r^h + \mathcal{O}(h^2),
\end{equation*}
which lead to
\begin{align*}
    \widetilde{A}_r^h - A_r^h = \frac{1}{h-\widetilde{I}_r^h-\widetilde{J}_r^h} - \frac{1}{h-I_r^h-J_r^h} &= \frac{-I_r^h-J_r^h +\widetilde{I_r^h}+\widetilde{J_r^h}}{(h-\widetilde{I}_r^h-\widetilde{J}_r^h)(h-I_r^h-J_r^h)} 
    \\ &  = \frac{\mathcal{O}(h^2)}{(h-\widetilde{I}_r^h-\widetilde{J}_r^h) (h-\widetilde{I}_r^h-\widetilde{J}_r^h + \mathcal{O}(h^2))}
    \\ &  = \mathcal{O}(1)
\end{align*}
since $\widetilde{I}_r^h = \mathcal{O}(h)$ and $\widetilde{J}_r^h = \mathcal{O}(h)$ by \eqref{eq:78} and \eqref{eq:79}. From this, we deduce, given that~$\widetilde{A}_r^h = \mathcal{O}(1/h)$,
\begin{equation*}
    B_r^h-\widetilde{B}_r^h = -A_r^h I_r^h + \widetilde{A}_r^h \widetilde{I}_r^h = -(\widetilde{A}_r^h + \mathcal{O}(1)) (\widetilde{I}_r^h + \mathcal{O}(h^2)) + \widetilde{A}_r^h \widetilde{I}_r^h = \mathcal{O}(h),
\end{equation*}
and similarly, 
\begin{equation*}
    C_r^h-\widetilde{C}_r^h = \mathcal{O}(h).
\end{equation*}
Thus, combining these results with the inductive hypothesis, we obtain
\begin{align*}
    P^1_{r+1}(t) & = \left(A^h_r t + B^h_r\right) P_r^1(t) - C^h_r P_{r-1}^1(t)
    \\ & = \left(\widetilde{A}^h_r t + \mathcal{O}(t) + \widetilde{B}^h_r + \mathcal{O}(h) \right) P_r^1(t) - (\widetilde{C}^h_r + \mathcal{O}(h)) P_{r-1}^1(t)
    \\ & = \left(\widetilde{A}^h_r t+  \widetilde{B}^h_r + \mathcal{O}(h) \right) (L_r(t/h) + \mathcal{O}(h)) - (\widetilde{C}^h_r + \mathcal{O}(h)) (L_{r-1}(t/h) + \mathcal{O}(h))
    \\ & = (\widetilde{A}^h_r t+  \widetilde{B}^h_r)  L_r(t/h)  - \widetilde{C}^h_r L_{r-1}(t/h) + \mathcal{O}(h)
    \\ & = L_{r+1}(t/h) + \mathcal{O}(h),
\end{align*}
and the proof is complete.
\end{proof}
\noindent
Lemma~\ref{lem:A2} ensures that the polynomials $\{P_r^i\}_{r \ge 0}$ share properties similar to those of the Legendre polynomials, which are used to derive the results of~Section~\ref{sec:33}, whose proofs have been postponed to Appendix~\ref{app:B}. We summarize these properties in the following proposition.
\begin{proposition} \label{prop:A3}
Let $\{P_r^i\}_{r \ge 0}$ the orthogonal polynomials in $[t_{i-1},t_i]$ with respect to a weight function $w(t)$ satisfying assumptions \eqref{eq:80}, and normalized such that $P_r^i(t_i)=1$. Moreover, assume there exists two constants~$w_0,w_1$ such that
\begin{equation} \label{eq:83}
   0\le w_0 \le w(t) \le w_1 \quad \text{for all~} t \in [0,T].
\end{equation}
Then, for all~$r\ge 0$, the following properties hold true:
\begin{align} \label{eq:84}
    & P_r^i(t_{i-1}) = (-1)^r + \mathcal{O}(h) & \text{as~} h \to 0,
    \\ \label{eq:85} & \int_{t_{i-1}}^{t_i} [P_r^i(t)]^2 w(t) \, \dd t = \mathcal{O}(h) & \text{as~} h \to 0,
    \\ \label{eq:86} & \int_{t_{i-1}}^{t_i} P_r^i(t) q_k(t) w(t) \, \dd t = \mathcal{O}(h^{k+1}) & \text{as~} h \to 0,
\end{align}
where $q_k(t) = (t-t_{i-1})^k$.
\end{proposition}
\begin{proof}
Using \eqref{eq:81}, we obtain $P_r^i(t_{i-1}) = P_r^1(0)$. Then, property \eqref{eq:84} follows from Lemma \ref{lem:A2} and the fact that $L_r(0) = (-1)^r$. To obtain \eqref{eq:85}, we compute with \eqref{eq:81} and \eqref{eq:80}
\begin{align*}
    \int_{t_{i-1}}^{t_i} [P_r^i(t)]^2 w(t) \dd t & = \int_{t_{i-1}}^{t_i} [P_r^1(t-t_{i-1})]^2 w(t) \dd t = w(t_{i-1}) \int_0^h [P_r^1(t)]^2 w(t) \dd t,
\end{align*}
and we conclude with \eqref{eq:82}, \eqref{eq:78} and assumption \eqref{eq:83}. Similarly, to deduce \eqref{eq:86}, we compute
\begin{align*}
    \int_{t_{i-1}}^{t_i} P_r^i(t) q_k(t) w(t) \dd t & = \int_{t_{i-1}}^{t_i} P_r^1(t-t_{i-1}) (t-t_{i-1})^k w(t) \dd t 
    \\ & = w(t_{i-1}) \int_0^h P_r^1(t) t^k w(t) \dd t
    \\ & = w(t_{i-1}) \int_0^h \left(L_r(t/h) + \mathcal{O}(h)\right) t^k (1+\mathcal{O}(t)) \, \dd t
    \\ & = w(t_{i-1}) \int_0^h L_r(t/h) t^k \, \dd t + \mathcal{O}(h^{k+2}),
\end{align*}
and we conclude with \eqref{eq:79} and \eqref{eq:83}.
\end{proof}

\section{Proofs of Propositions~\ref{prop:315} and~\ref{prop:316}}\label{app:B}

In this section, we collect the proofs of the two propositions stated in Section~\ref{sec:33}.

\subsection{Proof of Proposition~\ref{prop:315}}\label{app:B1}

The projection~$\mathcal{P}_h^p$ can also be equivalently defined  locally, for $i=1,\ldots,N$, by 
\begin{equation} \label{eq:87}
    \begin{cases}
        \mathcal{P}_h^p v (t_{i-1}^+) = \mathcal{P}_h^p v (t_{i-1}^-), &
        \\ (\mathcal{P}_h^p v, e^{-\bullet/T} q^{p-1})_{L^2(t_{i-1},t_i)} = (v, e^{-\bullet/T} q^{p-1})_{L^2(t_{i-1},t_i)} & \text{for all~} q^{p-1} \in \mathbb{P}^{p-1}(t_{i-1},t_i),
    \end{cases}
\end{equation}
where $\mathcal{P}_h^p v (t_i^{\pm}) := \lim_{t \to t_i^{\pm}} \mathcal{P}_h^p v (t)$, with the initial condition $\mathcal{P}_h^p v (t_0^-) =  v(0)$.

\noindent
Assume that $u_1$ and $u_2$ are two polynomials in~$\P^p(t_{i-1},t_i)$ satisfying~\eqref{eq:87}. Then, we have $u_1(t_{i-1}) = u_2(t_{i-1})$ and the difference $u_1-u_2$ can be expanded as 
\begin{equation*}
    u_1(t) - u_2(t) = \sum_{r=0}^p c_r P_r^i(t) \quad \text{for~} t \in [t_{i-1},t_i],
\end{equation*}
where $P_r^i$ are the orthogonal polynomials with respect to the weighted product~$(\cdot,e^{-\bullet/T} \cdot)_{L^2(t_{i-1},t_i)}$, normalized such that $P_r^i(t_i)=1$. For~$1\le r\le p,$ we have 
\begin{equation*}
    c_r: = \frac{(u_1-u_2, e^{-\bullet/T} P_r^i)_{L^2(t_{i-1},t_i)}}{(P_r^i, e^{-\bullet/T} P_r^i)_{L^2(t_{i-1},t_i)}}.
\end{equation*}
From the second condition in \eqref{eq:87}, we obtain $c_r = 0$ for $r=0,\ldots,p-1$. Furthermore, with the condition $u_1(t_{i-1}) = u_2(t_{i-1})$, we also conclude $c_p = 0$ since the zeros of an orthogonal polynomial lie strictly within the interval (see, e.g., \cite[Lemma 3.2]{Iserles2009}). This implies uniqueness, and the existence follows from finite dimensionality.

\noindent
Setting for all $i = 1,\ldots,N$, $\mathcal{P}_h^p v$ can be characterized as follows:
\begin{equation} \label{eq:88}
    \mathcal{P}_h^p v (t) = \Pi^{e,i}_{p-1} v (t) + \alpha_p^i P_p^i(t), \quad \text{for~} t \in [t_{i-1},t_i]   
\end{equation}
with $\Pi_{p-1}^{e,i} : L^2(t_{i-1}, t_i) \to \mathbb{P}^{p-1}(t_{i-1},t_i)$ is the projection 
with respect to the weighted scalar product $(\cdot,e^{-\bullet/T} \cdot)_{L^2(t_{i-1},t_i)}$ in $\mathbb{P}^{p-1}(t_{i-1},t_i)$, and $\alpha_p^i$ is chosen to ensure that $\mathcal{P}_h^p v \in C^0(0,T)$, namely, for $i=1,\ldots,N$ as
\begin{equation} \label{eq:89}
    \alpha_p^i = \bigl(P_p^i(t_{i-1})\bigl)^{-1} (\mathcal{P}_h^p v (t_{i-1}^-) - \Pi_{p-1}^{e,i} v (t_{i-1})).
\end{equation}
To prove~\eqref{eq:49}, using the local characterization in~\eqref{eq:88} (and~\eqref{eq:89}), we recursively obtain
\begin{equation*}
    \mathcal{P}_h^p v(t_i) = \Pi_{p-1}^{e,i} v (t_i) + \alpha_p^i P_p^i(t_i) = \Pi_{p-1}^{e,i} v (t_i) + \bigl(P_p^i(t_{i-1})\bigl)^{-1} (\mathcal{P}_h^p v (t_{i-1}) - \Pi_{p-1}^{e,i} v (t_{i-1})),
\end{equation*}
where we also used $P_p^i(t_i)=1$. For~$i=1,\ldots,N$, define 
\begin{equation} \label{eq:90}
    E^i_p(v):= \mathcal{P}_h^p v (t_i) - v(t_i). 
\end{equation}
In the rest of this proof, for brevity, we omit the dependence on~$v$ and we simply write~$E^i_p$. Note that, for $i=1,\ldots,N$,
\begin{equation} \label{eq:91}
\begin{aligned}
     P_p^i(t_{i-1}) E^i_p - E_p^{i-1} & = P_p^i(t_{i-1}) \mathcal{P}_h^p v(t_i) - P_p^i(t_{i-1}) v(t_i) - \mathcal{P}_h^p v(t_{i-1}) + v(t_{i-1})
    \\ & = P_p^i(t_{i-1}) \Pi_{p-1}^{e,i} v(t_i) - \Pi_{p-1}^{e,i} v(t_{i-1}) - P_p^i(t_{i-1}) v(t_i) + v(t_{i-1}) 
    \\ & = P_p^i(t_{i-1}) \left(\Pi_{p-1}^{e,i} v (t_i) - v(t_i)\right) - \left(\Pi_{p-1}^{e,i} v(t_{i-1}) - v(t_{i-1})\right).
\end{aligned}
\end{equation}
By performing a Taylor expansion around $t_{i-1}$, we obtain, for $t \in [t_{i-1},t_i]$,
\begin{equation} \label{eq:92}
    v(t) = \sum_{j=0}^{p+1} \frac{(t-t_{i-1})^j}{j!} \partial_t^j v(t_{i-1})  + \frac{(t-t_{i-1})^{p+2}}{(p+2)!}  \partial_t^{p+2} v(\xi_i(t)),
\end{equation}
for some~$\xi_i(t) \in [t_{i-1},t]$. Let us define $q_j(t) := (t-t_{i-1})^j$. Since~$\Pi_{p-1}^{e,i}q_j\equiv q_j$, we clearly have, for $j=0,\ldots,p-1$,
\begin{equation} \label{eq:93}
    P_p^i(t_{i-1}) \left(\Pi_{p-1}^{e,i} q_j(t_i) - q_j(t_i)\right) - \left(\Pi_{p-1}^{e,i} q_j(t_{i-1}) - q_j(t_{i-1})\right) = 0.
\end{equation}
Now, we show that~\eqref{eq:93} is also valid for $j=p$. First, as~$q_p(t_i)=h^p$ and~$q_p(t_{i-1})=0$, we have
\begin{equation} \label{eq:94}
\begin{aligned}
    & P_p^i(t_{i-1}) \left(\Pi_{p-1}^{e,i} q_p(t_i) - q_p(t_i) \right) - \left(\Pi_{p-1}^{e,i} q_p(t_{i-1}) - q_p(t_{i-1})\right)
    \\ & \hspace{2cm} =  P_p^i(t_{i-1}) \left(\Pi_{p-1}^{e,i} q_p(t_i) - h^p \right) - \Pi_{p-1}^{e,i} q_p(t_{i-1}).
\end{aligned}
\end{equation}
Note that we can characterize the projection of $q_p$ with
\begin{equation} \label{eq:95}
    \Pi^{e,i}_{p-1} q_p(t) = q_p(t) - \frac{(q_p, e^{-\bullet/T} P_p^i)_{L^2(t_{i-1},t_i)}}{(P^i_p, e^{-\bullet/T} P^i_p)_{L^2(t_{i-1},t_i)}} P_p^i(t),
\end{equation}
from which we deduce
\begin{equation} \label{eq:96}
    \Pi^{e,i}_{p-1} q_p(t_i) = h^p - \frac{(q_p, e^{-\bullet/T} P_p^i)_{L^2(t_{i-1},t_i)}}{(P_p^i, e^{-\bullet/T} P_p^i)_{L^2(t_{i-1},t_i)}} P_p^i(t_i) = h^p - \frac{(q_p, e^{-\bullet/T} P_p^i)_{L^2(t_{i-1},t_i)}}{(P_p^i, e^{-\bullet/T} P_p^i)_{L^2(t_{i-1},t_i)}}.
\end{equation}
We also compute
\begin{equation} \label{eq:97}
    \Pi^{e,i}_{p-1} q_p(t_{i-1}) = - \frac{(q_p, e^{-\bullet/T} P_p^i)_{L^2(t_{i-1},t_i)}}{(P^i_p, e^{-\bullet/T} P_p^i)_{L^2(t_{i-1},t_i)}} P_p^i(t_{i-1}).
\end{equation}
Inserting~\eqref{eq:96} and~\eqref{eq:97} into~\eqref{eq:94}, we conclude that \eqref{eq:93} is also valid for $j=p$. From this, by inserting \eqref{eq:92} into \eqref{eq:91}, we obtain
\begin{align*}
    P_p^i(t_{i-1}) E_p^i - E_p^{i-1}  = I_p^i + J_p^i
\end{align*}
with 
\begin{align*}
    I_p^i(v)&:= \frac{1}{(p+1)!} \partial_t^{p+1} v(t_{i-1}) \left( P_p^i(t_{i-1})  \left(\Pi_{p-1}^{e,i} q_{p+1}(t_i) - h^{p+1}\right) - \Pi_{p-1}^{e,i} q_{p+1}(t_{i-1}) \right) 
    \\ J_p^i(v) & := \frac{1}{(p+2)!} \left(P_p^i(t_{i-1}) \left(\Pi_{p-1}^{e,i} q_{p+2} \partial_t^{p+2} v(\xi_i)(t_i) - h^{p+2} \partial_t^{p+2} v (\xi_i(t_i)) \right) \right.
    \\ & \left. \hspace{9cm}- \Pi_{p-1}^{e,i} q_{p+2} \partial_t^{p+2} v(\xi_i) (t_{i-1}) \right).
\end{align*}
From the stability in $L^\infty$ of the weighted $L^2$-projection and property~\eqref{eq:84} of Proposition~\ref{prop:A3} in Appendix~\ref{app:A}, we obtain the following bound for~$J_p^i(v)$:
\begin{equation} \label{eq:98}
    |J_p^i(v)|\le Ch^{p+2}\norm{\partial_t^{p+2}v}_{L^{\infty}(t_{i-1},t_i)}.
\end{equation}
Here and throughout the proof,~$C$ denotes a positive constant only depending on~$p$ and~$T$, which may vary with each occurrence.
For the term inside the brackets in~$I_p^i(v)$, using the analogous expressions from~\eqref{eq:96} and~\eqref{eq:97} for $p+1$, along with Proposition \ref{prop:A3} in Appendix~\ref{app:A}, we obtain
\begin{align*}
    &P_p^i(t_{i-1}) \left(\Pi_{p-1}^{e,i} q_{p+1}(t_i) - h^{p+1}\right) - \Pi_{p-1}^{e,i} q_{p+1}(t_{i-1}) \\
    &\qquad =\frac{(q_{p+1}, e^{-\bullet/T} P_{p+1}^i)_{L^2(t_{i-1},t_i)}}{(P^i_{p+1}, e^{-\bullet/T} P_{p+1}^i)_{L^2(t_{i-1},t_i)}}\left(P_{p+1}^i(t_{i-1})-P_p^i(t_{i-1})\right)
    = C h^{p+1} + \mathcal{O}(h^{p+2}),
\end{align*}
where also the constant in the $\mathcal{O}$-notation depend only on $p$ and $T$. Therefore, we can write
\begin{equation*} 
    P_p^i(t_{i-1}) E_p^i - E_p^{i-1}  = Ch^{p+1} \partial_t^{p+1} v(t_{i-1}) + \mathcal{O}(h^{p+2}) \partial_t^{p+1} v(t_{i-1}) + J_p^i(v).
\end{equation*}
We use again property~\eqref{eq:84} of Proposition~\ref{prop:A3} to rewrite the left-hand side of the above equation and obtain
\begin{equation} \label{eq:99}
    (-1)^p E_p^i - E_p^{i-1} + \mathcal{O}(h) E^i_p = Ch^{p+1} \partial_t^{p+1} v(t_{i-1}) + \mathcal{O}(h^{p+2}) \partial_t^{p+1} v(t_{i-1}) + J_p^i(v).
\end{equation}
By proceeding recursively, taking into account that~$E_p^0=0$, we obtain
\begin{equation*}
\begin{aligned}
    |E_p^i| & \le \left(1+\mathcal{O}(h)\right)^{i} \left(\left(Ch^{p+1} + \mathcal{O}(h^{p+2})\right) \sum_{j=1}^{i} |\partial_t^{p+1} v(t_{j-1})| + \sum_{j=1}^i |J_p^{j}(v)
    |\right)
    \\ & =\left(1+i\,\mathcal{O}(h)\right)\left(\left(Ch^{p+1} + \mathcal{O}(h^{p+2})\right) \sum_{j=1}^{i} |\partial_t^{p+1} v(t_{j-1})| + \sum_{j=1}^i |J_p^{j}(v)
    |\right)\\
    & \le C h^p \left( \| \partial_t^{p+1} v\|_{L^\infty(0,T)} + \| \partial_t^{p+2} v\|_{L^\infty(0,T)}\right),
\end{aligned}
\end{equation*}
where, in the last step, we have also used~\eqref{eq:98}.
This estimate is sharp for even values of~$p$, leading to the bound in~\eqref{eq:49} for this case. However, 
when~$p$ is odd, we prove that an additional order of converge is achieved. Until this point, the fact that~$p$ is odd has not been used, but it becomes crucial at this point.

\noindent
Let us fix~$p$ odd. There exists~$D\in \R$ only depending on~$p$ and~$T$ such that~\eqref{eq:99} becomes
\begin{equation*}
     E_p^i + Dh E^i_p +E_p^{i-1} = -Ch^{p+1} \partial_t^{p+1} v(t_{i-1}) + \mathcal{O}(h^{p+2}) \partial_t^{p+1} v(t_{i-1}) - J_p^i(v).
\end{equation*}
Using $E_p^0=0$, we first compute
\begin{align*}
    (1+Dh)E_p^1 = -C h^{p+1} \partial_t^{p+1} v(0) + \mathcal{O}(h^{p+2}) \partial_t^{p+1} v(0) - J_p^1(v),
\end{align*}
which implies
\begin{equation*}
    E_p^1 = -C h^{p+1} \partial_t^{p+1} v(0)+h^{p+2}G_p^1(v),
\end{equation*}
with $G_p^1(v)$ such that, from~\eqref{eq:98}, for a constant~$C_G>0$ only depending on~$p$ and~$T$,
\begin{equation*}
    |G_p^1(v)| \le C_G \left( \| \partial_t^{p+1} v\|_{L^\infty(0,T)} + \| \partial_t^{p+2} v\|_{L^\infty(0,T)}\right).
\end{equation*}
Next, we obtain
\begin{equation} \label{eq:100}
\begin{aligned}
    (1+Dh) &E_p^2 = E_p^2+Dh E_p^2 +E_p^1-E_p^1\\
    & = C h^{p+1} (-\partial_t^{p+1} v(t_1) + \partial_t^{p+1} v(0)) + \mathcal{O}(h^{p+2}) \partial_t^{p+1} v(t_{1}) - J_p^2(v) -h^{p+2}G_p^1(v)\\ 
    & = -C h^{p+1} \int_0^{t_1} \partial_s^{p+2} v(s) \dd s + \mathcal{O}(h^{p+2}) \partial_t^{p+1} v(t_{1}) - J_p^2(v) -h^{p+2}G_p^1(v),
\end{aligned}
\end{equation}
which implies
\begin{equation*}
    E_p^2 = -C h^{p+1} \int_0^{t_1} \partial_s^{p+2} v(s) \dd s+h^{p+2}G_p^2(v),
\end{equation*}
with $G_p^2(v)$ such that
\begin{equation*}
    |G_p^2(v)| \le 2C_G \left( \| \partial_t^{p+1} v\|_{L^\infty(0,T)} + \| \partial_t^{p+2} v\|_{L^\infty(0,T)}\right).
\end{equation*}
Suppose now that, for an even index~$i$, we have
\begin{equation*}
    E_p^i = -C h^{p+1} \int_0^{t_i} \partial_s^{p+2} v(s) \dd s +h^{p+2}G_p^i(v),
\end{equation*}
with $G_p^i(v)$ such that
\begin{equation} \label{eq:101}
    |G_p^i(v)| \le i\,C_G \left( \| \partial_t^{p+1} v\|_{L^\infty(0,T)} + \| \partial_t^{p+2} v\|_{L^\infty(0,T)}\right).
\end{equation}
Then, for the next odd index~$i+1$, we deduce
\begin{align*}
    (1+Dh)E_p^{i+1} &= E_p^{i+1}+Dh E_p^{i+1} +E_p^{i}-E_p^{i}\\ 
    & = C h^{p+1} \left(-\partial_t^{p+1} v(t_i)+ \int_0^{t_i} \partial_s^{p+2} v(s)\dd s\right)\\
    & \qquad + \mathcal{O}(h^{p+2}) \partial_t^{p+1} v(t_{i}) - J_p^{i+1}(v) -h^{p+2}G_p^i(v),
\end{align*}
which implies
\begin{equation} \label{eq:102}
    E_p^{i+1} = C h^{p+1} \left(-\partial_t^{p+1} v(t_i)+ \int_0^{t_i} \partial_s^{p+2} v(s)\dd s\right)+h^{p+2}G_p^{i+1}(v),
\end{equation}
with~$G_p^{i+1}(v)$ satisfying~\eqref{eq:101} with~$i+1$ instead of~$i$.
For the next even index~$i+2$, we have
\begin{equation} \label{eq:103}
\begin{aligned}
    (1+Dh)E_p^{i+2} & = E_p^{i+2} +DhE_p^{i+2} + E_p^{i+1} - E_p^{i+1}
    \\ & = C h^{p+1} \left( -\partial_t^{p+1} v(t_{i+1}) + \partial_t^{p+1} v(t_i) - \int_0^{t_i} \partial_s^{p+2} v(s) \dd s \right) \\
    & \qquad + \mathcal{O}(h^{p+2}) \partial_t^{p+1} v(t_{i+1}) - J_p^{i+2}(v) -h^{p+2}G_p^{i+1}(v)\\
    & = - C h^{p+1} \int_0^{t_{i+1}} \partial_s^{p+2} v(s) \dd s 
    \\ & \qquad + \mathcal{O}(h^{p+2}) \partial_t^{p+1} v(t_{i+1}) - J_p^{i+2}(v) -h^{p+2}G_p^{i+1}(v),
\end{aligned}
\end{equation}
which implies
\begin{equation} \label{eq:104}
    E_p^{i+2} = -C h^{p+1} \int_0^{t_{i+1}} \partial_s^{p+2} v(s) \dd s +h^{p+2}G_p^{i+2}(v),
\end{equation}
with $G_p^{i+2}(v)$ satisfying~\eqref{eq:101} with~$i+2$ instead of~$i$. Therefore, combining~\eqref{eq:102} and~\eqref{eq:104} with~\eqref{eq:101}, recalling the definition~\eqref{eq:90}, completes the proof of~\eqref{eq:49} also for $p$ odd.\\
\noindent
We emphasize that the terms containing~$\partial_t^{p+1}v(t_{i+1})$ and $\partial_t^{p+1}v(t_i)$ appear in the second lines of \eqref{eq:100} and \eqref{eq:103} with alternating signs, which enables their reorganization into integral terms in the last equalities of these equations. This leads to an additional order of convergence compared to the case of even $p$, where the terms containing~$\partial_t^{p+1}v(t_{i+1})$ and $\partial_t^{p+1}v(t_i)$ appear with the same sign and can only be estimated from above by the sum of their absolute values.

\subsection{Proof of Proposition~\ref{prop:316}}\label{app:B2}
Let $i \in \{1,\ldots,N\}$ be fixed. We write
\begin{equation} \label{eq:105}
    \| v-\mathcal{P}_h^p v \|_{L^2(t_{i-1},t_i)} \le \| v - \Pi_p^{e,i} v \|_{L^2(t_{i-1},t_i)} + \| \mathcal{P}_h^p v - \Pi_p^{e,i} v \|_{L^2(t_{i-1},t_i)},
\end{equation}
where, as in the proof of Proposition~\ref{prop:315},~$\Pi_{p-1}^{e,i} : L^2(t_{i-1}, t_i) \to \mathbb{P}^{p-1}(t_{i-1},t_i)$ denotes the projection with respect to the weighted scalar product $(\cdot,e^{-\bullet/T} \cdot)_{L^2(t_{i-1},t_i)}$ in $\mathbb{P}^{p-1}(t_{i-1},t_i)$. From standard approximation results, the second term on the right-hand side of~\eqref{eq:105} satisfies
\begin{equation} \label{eq:106}
    \| v -  \Pi_p^{e,i} v \|_{L^2(t_{i-1},t_i)} \le C h^{p+1} \| \partial_t^{p+1} v \|_{L^2(t_{i-1},t_i)} \le C h^{p+3/2}\| \partial_t^{p+1} v \|_{L^\infty(0,T)}.
\end{equation}
Here and in the rest of this proof,~$C>0$ denote a constant only depending on~$p$ and~$T$, which may change at each occurrence. To estimate the first term, we explicitly compute, for $t \in [t_{i-1},t_i]$, using \eqref{eq:88} and \eqref{eq:89},
\begin{equation*}
    \mathcal{P}_h^p v(t) - \Pi_p^{e,i} v (t) = \Pi^{e,i}_{p-1} v (t) + \bigl(P_p^i(t_{i-1})\bigr)^{-1} (\mathcal{P}_h^p v (t_{i-1}) - \Pi_{p-1}^{e,i} v (t_{i-1})) P_p^i(t) - \Pi_p^{e,i} v(t),
\end{equation*}
where, as in the proof of Proposition~\ref{prop:315},~$P_r^i$ are orthogonal polynomials with respect to the weighted product~$(\cdot,e^{-\bullet/T} \cdot)_{L^2(t_{i-1},t_i)}$. For the difference~$\Pi_{p-1}^{e,i} v(t) - \Pi_p^{e,i} v(t)$, we have
\begin{equation*}
    \Pi_{p-1}^{e,i} v(t) - \Pi_p^{e,i} v(t) = - \frac{(v,e^{-\bullet/T}P^i_p)_{L^2(t_{i-1},t_i)}}{(P^i_p, e^{-\bullet/T}P^i_p)_{L^2(t_{i-1},t_i)}}
    P_p^i(t).
\end{equation*}
From this, in the weighted norm $\|\cdot\|^2_{L^2_e(t_{i-1},t_i)} := (\cdot,e^{-\bullet/T} \cdot)_{L^2(t_{i-1},t_i)}$, we obtain
\begin{align*}
    & \| \mathcal{P}_h^p v - \Pi_p^{e,i} v \|^2_{L^2_e(t_{i-1},t_i)} 
    \\ & \hspace{1cm} = \left|\frac{(v,e^{-\bullet/T}P^i_p)_{L^2(t_{i-1},t_i)}}{(P^i_p, e^{-\bullet/T}P^i_p)_{L^2(t_{i-1},t_i)}}- \bigl(P_p^i(t_{i-1})\bigr)^{-1}\bigl(\mathcal{P}_h^p v (t_{i-1}) - \Pi_{p-1}^{e,i} v (t_{i-1})\bigr) \right|^2 \| P_p^i \|^2_{L^2_e(t_{i-1},t_i)}.
\end{align*}
From \eqref{eq:85} in Proposition~\ref{prop:A3}, we get $\|P_p^i\|^2_{L^2_e(t_{i-1},t_i)} = \mathcal{O}(h)$, from which
\begin{align*}
    \| \mathcal{P}_h^p v &- \Pi_p^{e,i} v \|^2_{L^2_e(t_{i-1},t_i)}\\ &= C h  \left|\frac{(v,e^{-\bullet/T}P^i_p)_{L^2(t_{i-1},t_i)}}{(P^i_p, e^{-\bullet/T} P^i_p)_{L^2(t_{i-1},t_i)}} - \bigl(P_p^i(t_{i-1})\bigr)^{-1}\bigl(\mathcal{P}_h^p v (t_{i-1}) - \Pi_{p-1}^{e,i} v (t_{i-1})\bigr) \right|^2\\
    &\le Ch \left|\frac{(v,e^{-\bullet/T} P^i_p)_{L^2(t_{i-1},t_i)}}{(P^i_p, e^{-\bullet/T} P^i_p)_{L^2(t_{i-1},t_i)}} - \bigl(P_p^i(t_{i-1})\bigr)^{-1}\left(v(t_{i-1}) - \Pi_{p-1}^{e,i} v (t_{i-1})\right) \right|^2 
    \\ & \hspace{5.2cm} + Ch \bigl|P_p^i(t_{i-1})\bigr|^{-2}
    \left| \mathcal{P}_h^p v (t_{i-1}) - v(t_{i-1}) \right|^2. 
\end{align*}
Then, from Proposition~\ref{prop:315} and Proposition~\ref{prop:A3}, we obtain
\begin{equation} \label{eq:107}
\begin{aligned}
    \| \mathcal{P}_h^p v & - \Pi_p^{e,i} v \|^2_{L^2_e(t_{i-1},t_i)} 
    \\ & \le Ch \left|\frac{(v,e^{-\bullet/T}P^i_p)_{L^2(t_{i-1},t_i)}}{(P^i_p,e^{-\bullet/T}P^i_p)_{L^2(t_{i-1},t_i)}} - \bigl(P_p^i(t_{i-1})\bigr)^{-1}\bigr(v (t_{i-1}) - \Pi_{p-1}^{e,i} v (t_{i-1})\bigr) \right|^2 
    \\ & \hspace{2.5cm} + C h^{2p+1+2(p \hspace{-0.2cm} \mod 2)} \left(\| \partial_t^{p+1} v \|_{L^\infty(0,T)} + \| \partial_t^{p+2} v\|_{L^\infty(0,T)}\right)^2.
\end{aligned}
\end{equation}
It remains to estimate the first term on the right-hand side of~\eqref{eq:107}. We expand $v$ in Taylor series around $t_{i-1}$ up to order $p$ (see~\eqref{eq:92} with $p$ instead of~$p+1$), and set $q_j(t) := (t-t_{i-1})^j$. Due to orthogonality properties, we have
\begin{equation*}
    \frac{(q_j,e^{-\bullet/T} P^i_p)_{L^2(t_{i-1},t_i)}}{(P^i_p,e^{-\bullet/T}P^i_p)_{L^2(t_{i-1},t_i)}} - \bigl(P_p^i(t_{i-1})\bigr)^{-1} \bigl(q_j (t_{i-1}) - \Pi_{p-1}^{e,i} q_j (t_{i-1})\bigr) = 0, \quad j = 0,\ldots,p-1.
\end{equation*}
This is also true for $j=p$, as can be seen from~\eqref{eq:95}.  Therefore, only the remainder term gives a contribution. Taking into account that~$q_{p+1}(t_{i-1})=0$, this term is equal to
\begin{equation*}
    \frac{1}{(p+1)!}\left( \frac{(q_{p+1}\partial^{p+1}_t v(\xi_i),e^{-\bullet/T} P^i_p)_{L^2(t_{i-1},t_i)}}{(P^i_p,e^{-\bullet/T}P^i_p)_{L^2(t_{i-1},t_i)}}+\bigl(P_p^i(t_{i-1})\bigr)^{-1}\Pi_{p-1}^{e,i}q_{p+1}\partial^{p+1}_t v(\xi_i)(t_{i-1})
\right).
\end{equation*}
The latter contribution can be estimated using the stability of the weighted $L^2$ projection in the $L^\infty$ norm, along with Proposition \ref{prop:A3}. This leads to the estimate
\begin{equation*}
    \left|\frac{(v,e^{-\bullet/T} P^i_p)_{L^2(t_{i-1},t_i)}}{(P^i_p,e^{-\bullet/T}P^i_p)_{L^2(t_{i-1},t_i)}} - \bigl(P_p^i(t_{i-1})\bigr)^{-1}\bigl(v (t_{i-1}) - \Pi_{p-1}^{e,i} v (t_{i-1})\bigr) \right| \le C h^{p+1} \| \partial_t^{p+1} v \|_{L^\infty(t_{i-1},t_i)}.
\end{equation*}
Combining this result with~\eqref{eq:107}, and using the equivalence of the weighted norm with the standard~$L^2$ norm give
\begin{align*}
    \| \mathcal{P}_h^p v - \Pi_p^{e,i} v \|_{L^2(t_{i-1},t_i)}&\le C\| \mathcal{P}_h^p v - \Pi_p^{e,i} v \|_{L^2_e(t_{i-1},t_i)} \\
    &\le C h^{p+1/2+(p\hspace{-0.2cm} \mod 2)} \left( \| \partial_t^{p+1} v \|_{L^\infty(0,T)} + \| \partial_t^{p+2} v \|_{L^\infty(0,T)} \right).
\end{align*}
Finally, by inserting this and~\eqref{eq:106} into~\eqref{eq:105}, squaring both sides and summing over all intervals, we obtain the desired result.

\end{document}